\documentclass[10pt,reqno]{article}
\usepackage{amsmath,amssymb,amsthm}
\usepackage{esint}
\usepackage{bm}
\usepackage{subfigure}
\usepackage{graphicx,color}
\usepackage[sort&compress]{natbib}
\usepackage{algorithm}
\usepackage{algpseudocode}
\usepackage{algorithmicx}

\DeclareMathAlphabet{\itbf}{OML}{cmm}{b}{it}

\def\by{{{\itbf y}}}
\def\bx{{{\itbf x}}}
\def\bz{{{\itbf z}}}

\def\etab{{\boldsymbol{\eta}}_{\rm noise}}
\def\bpsi{{\boldsymbol{\psi}}}

\newcommand{\bU}{\mathbf{U}}
\newcommand{\RR}{\mathbb{R}}

\newcommand{\K}{{\kappa}}

\newcommand{\bu}{\mathbf{u}}
\newcommand{\bv}{\mathbf{v}}
\newcommand{\bw}{\mathbf{w}}

\newcommand{\ds}{\displaystyle}

\newtheorem{thm}{Theorem}[section]
\newtheorem{cor}[thm]{Corollary}
\newtheorem{lem}[thm]{Lemma}

\newtheorem{defn}[thm]{Definition}

\numberwithin{equation}{section}

\newcommand{\pathfigures}{Figures/}
\graphicspath{{\pathfigures}}

\begin{document}

\title{Stability and Resolution Analysis of Topological Derivative Based Localization of Small Electromagnetic
Inclusions}
\author{
Abdul Wahab
\thanks{\footnotesize Department of Mathematics, COMSATS Institute of Information Technology, 47040, Wah Cantt., Pakistan (wahab@ciitwah.edu.pk).}
}
\maketitle

\begin{abstract}
The aim of this article is to elaborate and rigorously analyze a topological derivative based imaging framework for locating an electromagnetic inclusion of diminishing size from boundary measurements of the tangential component of scattered magnetic field at a fixed frequency.  The inverse problem of inclusion detection is formulated as an optimization problem in terms of a filtered discrepancy functional and the topological derivative based imaging functional obtained therefrom. The sensitivity and resolution analysis of the imaging functional is rigorously performed. It is substantiated that the Rayleigh resolution limit is achieved. Further, the stability of the reconstruction with respect to measurement and medium noises is investigated and the signal-to-noise ratio is evaluated in terms of the imaginary part of free space fundamental magnetic solution.
\end{abstract}



\section{Introduction} 

The concept of derivatives with respect to geometry or topology has played a significant role in industrial and engineering  optimization problems, especially for designing optimal shapes of various products subject to industrial constraints \cite{SZ}. Soon after its emergence \cite{EKS}, the idea was embraced for  imaging of diametrically small anomalies \cite{CGGM}  and inverse scattering problems; see, for example, \cite{Princeton, Multi, BG, DG, DGE, F, HL} and articles cited therein. 

In topological derivative based imaging framework, a trial inclusion is created in the (inclusion-free) background medium  at a search point, furnishing fitted data. Then a misfit functional is constructed using measurements and the fitted data. The search points that minimize the discrepancy between measured data and the fitted data are then sought. In order to find its minima, the misfit is expanded using the asymptotic expansions due to the perturbation of the wave-field in the presence of an inclusion versus its characteristic size. The leading order term in the expansion is then referred to as the topological derivative of the misfit, which synthesizes its
sensitivity relative to the insertion of an inclusion at a given search location. Its maximum, which corresponds to the point at which the insertion of the inclusion maximally decreases the misfit is therefore a potential candidate for the location of the true inclusion.

The topological derivatives have been used heuristically in the context of imaging and non-destructive testing lacking rigorous mathematical justifications, unlike in shape optimization wherein they attracted enormous interest from mathematical as well as numerical view point. For the first time, the stability and resolution analysis of the topological derivative based imaging of small inclusions for the anti-plane elasticity was performed by  \citet{AGJK}. Therein, it is elucidated that in order to get a stable and guaranteed localization with a good resolution, the use of a filtered discrepancy is indispensable whereas the filter needs to be defined in terms of a \emph{Neumann-Poincar\'e} type boundary integral operator. The filtered topological derivative functional is proved to achieve Raleigh resolution limit. Moreover, it is elucidated that this topological sensitivity framework is stable and robust with respect to medium and measurement noises, and with limited view measurements. It performs far batter than classical imaging frameworks including back-propagation technique, MUSIC-type imaging and Kirchhoff migration in  worse imaging conditions. 

The full elasticity  case of topological sensitivity framework in a linear isotropic regime was rigorously explained by \citet{TDelastic}. The study surprisingly indicates that the classical framework does not guarantee a localization of the inclusion even with a filtered discrepancy functional. Moreover, even if it is somehow able to locate the inclusion, the resolution of the functional degenerates thanks to nonlinear coupling between shear and pressure components at the boundary. In order to counter the coupling artifacts and to have a guaranteed localization of small inclusions, a modified imaging framework was proposed based on a weighted Helmholtz decomposition \cite{ABGW} applied to the initial guess furnished by filtered topological derivative functional. The modified framework is then proved to be stable with respect to medium and measurement noises. Furthermore, it achieves the Rayleigh resolution limit.  

The aim in this article is to study a topological derivative based imaging framework for detecting diametrically small electromagnetic inclusions from  single and multiple boundary measurements of the tangential component of scattered magnetic field over a fixed frequency. It is assumed that the magnetic field satisfies full three dimensional Maxwell equations and the inclusion is penetrable however homogeneous with electromagnetic parameters different from that of the background medium.  The work is focused on the analysis of the detection capabilities of a filtered topological derivative based imaging functional wherein the filter is defined in terms of a boundary integral operator. Precisely, the aim of the article is three-fold: First to introduce a filtered topological derivative based imaging framework,  then to perform sensitivity and resolution analysis of the algorithm and finally to investigate its stability with respect to measurement and medium noises. 
The potential applications envisioned by the imaging of electromagnetic inclusions of diminishing size can be found in non-destructive testing of small material impurities, medical diagnosis and therapeutic protocols, especially for detecting and curing cancers of vanishing size and for brain imaging. It is worthwhile precising that the problem of detecting small electromagnetic inclusions has been previously studied by using MUSIC-type algorithms \cite{AILP}, time reversal and phase conjugation techniques \cite{GWL, WARS, WAHR}, reverse time migration \cite{CCH}, topological derivative based imaging \cite{MPS}, and asymptotic expansion techniques \cite{AKInv, AK}. For the imaging of thin electromagnetic inclusions and cracks in a two dimensional setting, we refer the reader to \cite{park, park2} for instance. We will restrict ourselves only to the detection of the inclusion and will not discuss its morphology (shape, size and material properties) in this paper. In this regard, we refer for instance to the recent results by \citet{AM} and \citet{Bao}.

The rest of this article is organized in the following manner. In Section \ref{form}, we collect some notation and important results on electromagnetic Green's functions, boundary layer potentials and polarization tensors. The inverse problem under taken in this study is then mathematically formulated. In Section \ref{framework}, a filtered quadratic misfit is defined and its topological derivative is evaluated using asymptotic expansion of the scattered magnetic field with respect to the characteristic size of the inclusion. The sensitivity and resolution  analysis of the imaging functional is performed in Section \ref{sensitivity}. Section \ref{measNoise} is dedicated to perform stability analysis of the topological derivative based imaging with respect to measurement noise whereas Section \ref{medNoise} deals with its stability with respect to medium noise.  Finally, a summary of the results obtained herein is provided in Section \ref{conc}.

\section{Mathematical formulation}\label{form}

In this section, we introduce some notation and collect some basic results for electromagnetic Green's functions and layer potentials indispensable for this study. We also mathematically formulate the inverse problem undertaken. 

\subsection{Notation}

Let $X\subset\RR^3$ be a smooth domain with simply connected boundary ${\partial X}$ and $\nu$ denote the outward unit normal vector on ${\partial X}$.  
We define the surface divergence  of a complex valued vector field $\bu\in\mathcal{C}^k({\partial X})$ for $k\in\mathbb{N}$  by 
\begin{eqnarray}
{\rm div}_{\partial X} \bu = {\rm div}\widetilde{\bu}|_{\partial X} -\left(\big[\nabla\widetilde{\bu}|_{\partial X}\big]\nu\right)\cdot \nu,
\end{eqnarray} 
where $\widetilde{\bu}$ is a smooth extension of $\bu$ to the whole space $\RR^3$.

Let $H^s(X)$ and $H^s_{\rm loc}(\overline{X})$ be the usual Sobolev spaces for $s>0$. By $H^s({\partial X})$ the trace space of $H^{s+1/2}(X)$ and by $H^{-s}$ the $L^2-$dual space of $H^s({\partial X})$ are denoted. Moreover, $TH^s({\partial X})$ defines the tangential trace space of $H^{s+1/2}(X)$ under the action of the operator $\gamma_\nu[\bu]=\nu\times\bu|_{\partial X}$ and its $L^2-$ dual is denoted by $TH^{-s}({\partial X})$.  We also define the Hilbert space
\begin{eqnarray}
TH_{\rm div}^s({\partial X}):=\left\{\bu\in \left(TH^s ({\partial X})\right)^3\quad \big|\quad {\rm div}_{\partial X}\bu\in H^s({\partial X})\right\}.
\end{eqnarray}
Similarly, the dual space of $TH_{\rm div}^s({\partial X})$ is denoted by $TH_{\rm div}^{-s}({\partial X})$. Finally, we define the spaces $H(X;{\rm{\bf curl}})$ and $H_{\rm loc}(X; {\rm{\bf curl}})$  by 
\begin{eqnarray}
H(X;{\rm{\bf curl}})&:=&\left\{\bu\in\left(H^s(X)\right)^3\quad\Big|\quad {\rm{\bf curl}}\,\bu\in L^2(X)\right\},
\\
H_{\rm loc}(X;{\rm{\bf curl}})&:=&\left\{\bu\in\left(H^s_{\rm loc}(X)\right)^3\quad\Big|\quad {\rm{\bf curl}}\,\bu\in L^2_{\rm loc}(X)\right\}.
\end{eqnarray}
Refer to \cite{nedelec, BC, BHPS} and references therein for further details. 

For matrices $\mathbf{A}=(a_{ij})_{i,j=1}^3$ and $\mathbf{B}=(b_{ij})_{i,j=1}^3$, the contraction operator `$:$' is defined by $\mathbf{A}:\mathbf{B} := \sum_{i,j=1}^3 a_{ij}b_{ij},$ and the \emph{Frobenius norm} $\|\cdot\|$ of   $\mathbf{A}$ is defined by  $\|\mathbf{A}\|:= \sqrt{\mathbf{A}:\mathbf{A}}.$

\subsection{Problem formulation}

Let $D=\rho B_D+ \bz_D$ be a small three-dimensional bounded inclusion with a smooth and simply connected boundary $\partial D$,  permittivity $\epsilon_1>0$ and permeability $\mu_1>0$, where $B_D$ is a regular enough bounded domain in $\RR^3$ representing the volume of the inclusion, $\bz_D$ is the vector position of its center and $\rho>0$ is the scale factor. The inclusion $D$ is compactly supported in the bounded open background domain $\Omega\subset\RR^3$ with a smooth and simply connected boundary $\partial\Omega$. Let $\epsilon_0>0$ and $\mu_0>0$ be the permittivity and  permeability of  $\Omega$ without inclusion $D$, letting $\K:=\omega\sqrt{\epsilon_0\mu_0}$ and $c=1/\sqrt{\epsilon_0\mu_0}$ to be the background wave-number and  speed of light in the medium, respectively, where $\omega>0$ is the frequency pulsation. We define the piecewise constant functions $\mu_\rho$ and $\epsilon_\rho$ by 
\begin{eqnarray}
\mu_\rho(\bx):=
\begin{cases} 
\mu_1, & \bx\in D, 
\\
\mu_0, & \bx\in\RR^3\setminus \overline{D},
\end{cases}
\quad\text{and}\quad 
\epsilon_\rho(\bx)
:=
\begin{cases} 
\epsilon_1, & \bx\in D, 
\\
\epsilon_0, & \bx\in \RR^3\setminus \overline{D}.
\end{cases}
\end{eqnarray}
Furthermore, let $D$ be of diminishing characteristic size and be separated apart from $\partial\Omega$, that is, there exists a constant $d_0>0$ such that
\begin{eqnarray}
\inf_{\bx\in D}{\rm dist}(\bx,\partial\Omega)\geq d_0>0
\quad\text{and}\quad 
\rho\K\ll 1.\label{assumption}
\end{eqnarray}

Let $\mathbf{H}_\rho\in H_{\rm loc}(\rm{\bf{curl}}, \Omega)$ denote the time-harmonic magnetic field in $\Omega$ in the presence of $D$, that is, the solution to 
\begin{eqnarray}
\begin{cases}\label{Hrho}
\ds\nabla\times (\epsilon_0^{-1}\nabla\times \mathbf{H}_\rho)-\omega^2\mu_0\mathbf{H}_\rho= \mathbf{0}, & \Omega\setminus\overline D,
\\
\ds\ds\nabla\times (\epsilon_1^{-1}\nabla\times \mathbf{H}_\rho)-\omega^2\mu_1\mathbf{H}_\rho= \mathbf{0}, & D,
\\
\ds\left(\mathbf{H}_\rho\times\nu\right)^+ - \left(\mathbf{H}_\rho\times\nu \right)^-=\mathbf{0}, 
&\partial D,
\\
\ds \epsilon_0^{-1}(\nabla\times\mathbf{H}_\rho)^+\times\nu-\ds \epsilon_1^{-1}(\nabla\times\mathbf{H}_\rho)^-\times\nu=\mathbf{0},
 &\partial D,
\\
\ds \mu_0 \mathbf{H}_\rho^+\cdot\nu
-\mu_1\mathbf{H}_\rho^-\cdot\nu=\mathbf{0}, &\partial D,
\\
\epsilon_0^{-1}(\nabla\times\mathbf{H}_\rho)\times \nu = \mathbf{h}, & \partial \Omega.
\end{cases}
\end{eqnarray}

We also define the background magnetic field (in the absence of any inclusion inside $\Omega$) $\mathbf{H}_0\in H_{\rm loc}(\rm{\bf{curl}}, \Omega)$ as the solution to 
\begin{eqnarray}\label{H0}
\begin{cases}
\ds\nabla\times\nabla\times \mathbf{H}_0-\K^2\mathbf{H}_0= \mathbf{0}, & \Omega,
\\
\epsilon_0^{-1}(\nabla\times\mathbf{H}_0)\times \nu = \mathbf{h}, & \partial \Omega.
\end{cases}
\end{eqnarray}

In this paper, we are interested in the following problem.
\subsubsection*{Inverse problem} 
\textit{Given the measurements $\mathbf{H}_\rho\times \nu$ for all $\bx\in\partial \Omega$, find the position $\bz_D$ of the inclusion $D$ using a filtered topological derivative based imaging framework.}

A similar problem has been studied by \citet{MPS} using topological derivative based sensitivity framework by invoking an adjoint field. The aim here is to design and debate the performance of  topological derivative based detection framework applied to a filtered quadratic misfit. Moreover, the approach adopted herein is based on the asymptotic expansion of the scattered magnetic field with respect to the size of the inclusion.

\subsection{Electromagnetic Green's functions} 

Consider the outgoing fundamental solution $g$ to Helmholtz operator $-(\Delta+\K^2)$ in $\RR^3$ given by 
\begin{eqnarray}
g(\bx,\by):=\frac{e^{i\K|\bx-\by|}}{4\pi|\bx-\by|}, \qquad \bx\neq\by,\quad  \bx,\by\in\RR^3,
\end{eqnarray}
and introduce the dyadic Green's function $\mathbf{\Gamma}$ by 
\begin{eqnarray}
\mathbf{\Gamma}(\bx,\by):= -\epsilon_0\left(\mathbf{I}_3+\frac{1}{\K^2}\nabla\nabla^T\right) g(\bx,\by),
\end{eqnarray}
where $\mathbf{I}_3$ is $3\times 3$ identity matrix. The function $\mathbf{\Gamma}(\bx,\by)$ is the solution to 
\begin{eqnarray}
\ds\nabla_\bx\times \nabla_\bx \times \mathbf{\Gamma}(\bx,\by)-\K^2 \mathbf{\Gamma}(\bx,\by)= -\epsilon_0\delta_\by(\bx)\mathbf{I}_3, &\bx,\by\in\RR^3,
\end{eqnarray}
subject to \emph{Silver-M\"uller} condition
\begin{eqnarray}
\ds\lim_{|\bx-\by|\to \infty} |\bx-\by|\left[\nabla_\bx\times\mathbf{\Gamma}(\bx,\by)\times\frac{\bx-\by}{|\bx-\by|}-i\K\mathbf{\Gamma}(\bx,\by)\right]=0.\label{silver}
\end{eqnarray}
Here $\delta_{\by}(\cdot)=\delta_0(\cdot-\by)$ is the Dirac mass at $\by$ and the operator $\nabla\times$ acts on matrices column-wise, that is, $$
\nabla\times \left[\mathbf{\Gamma}\mathbf{p}\right]:= \nabla\times \left[\mathbf{\Gamma}\right]\mathbf{p}, \quad \text{for all constant vectors}\quad \mathbf{p}\in\RR^3.
$$
It is worthwhile precising that $\mathbf{\Gamma}$ possesses the following reciprocity properties in isotropic dielectric materials; see \cite[Sect. 2.2]{AILP},
\begin{eqnarray}
\mathbf{\Gamma}(\bx,\by)=  \mathbf{\Gamma}(\by,\bx) 
\quad\text{and}\quad
\nabla_\bx\times \mathbf{\Gamma}(\bx,\by)=\left[\nabla_\by\times \mathbf{\Gamma}(\by,\bx)\right]^T.\label{Pro1}
\end{eqnarray}

The following electromagnetic Helmholtz-Kirchhoff identities are the key ingredients to elucidate the localization capabilities of the imaging functional proposed in the next section. 

\begin{lem}[See {\cite[Lemma 3.2]{CCH}} ]\label{HKI-lem1}
Let $\mathbf{B}(0,r)$ be an open ball in $\RR^3$ with large radius $r\to\infty$ and boundary $\partial\mathbf{B}(0,r)$. Then, for all $\bx,\by\in\mathbf{B}(0,r)$, we have
\begin{align}\label{HK1}
\int_{\partial\mathbf{B}(0,r)}\left(\overline{\mathbf{\Gamma}(\bx,\bz)}\right)^T\mathbf{\Gamma}(\bz,\by)d\sigma(\bz) = -\frac{\epsilon_0}{\K}\Im m\big\{\mathbf{\Gamma}(\bx,\by)\big\}+\mathbf{Q}(\bx,\by),
\end{align}
where $\mathbf{Q}=\left(q_{ij}\right)_{i,j=1}^3$ is such that 
$|q_{ij}(\bx,\by)|+|\nabla_\bx q_{ij}(\bx,\by)|\leq Cr^{-1}$ uniformly for all $\bx,\by\in\mathbf{B}(0,r)$. Here and throughout this paper $d\sigma$ denotes the surface element. 
\end{lem} 

\begin{lem}\label{HKI-lem2}
Let $\mathbf{B}(0,r)$ be an open ball in $\RR^3$ with large radius $r\to\infty$ and boundary $\partial\mathbf{B}(0,r)$. Then, for all $\bx,\by\in\mathbf{B}(0,r)$, we have
\begin{equation}\label{HK2}
\int_{\partial\mathbf{B}(0,r)}\Big(\overline{\mathbf{\Gamma}(\bx,\bz)}\times\nu(\bz)\Big)^T\Big(\mathbf{\Gamma}(\bz,\by)\times\nu(\bz)\Big)d\sigma(\bz) = -\frac{\epsilon_0}{\K}\Im m\big\{\mathbf{\Gamma}(\bx,\by)\big\}+\widetilde{\mathbf{Q}}(\bx,\by),
\end{equation}
where $\widetilde{\mathbf{Q}}=\left(\widetilde{q}_{ij}\right)_{i,j=1}^3$ is such that 
$|\widetilde{q}_{ij}(\bx,\by)|+|\nabla_\bx \widetilde{q}_{ij}(\bx,\by)|\leq \widetilde{C}r^{-1}$ uniformly for all $\bx,\by\in\mathbf{B}(0,r)$.
\end{lem} 

The identity \eqref{HK2} can be proved trivially by mimicking the proof of Lemma \ref{HKI-lem1} provided in \cite[Lemma 3.2]{CCH}. For the sake of completeness, we briefly sketch the proof in Appendix \ref{Append.HK2}. 

\subsection{Layer potentials}

We define the scalar single layer potential $S^\K$ associated with domain $X$ of a scalar field $\phi\in H^{s-1/2}({\partial X})$ by
\begin{equation}
S^{\K}[\phi](\bx):=\int_{{\partial X}} g(\bx,\by)\phi(\by)d\sigma(\by),\qquad \bx\in\RR^3\setminus {\partial X}.
\end{equation}
The vector single layer potential is defined likewise and still represented by $S^\K$ by abuse of notation. We have the following result from \cite[Lemma 2.3]{MS}.
\begin{lem}\label{LemDivS}
For all $\mathbf{j}\in TH_{\rm div}^{-1/2}({\partial X})$
\begin{align}
&\nabla\cdot S^{\K}[\,\mathbf{j}\,](\bx)= S^{\K}[{\rm div}_{{\partial X}}\mathbf{j}](\bx), \quad \bx\in\RR^3\setminus{\partial X}.
\end{align}
\end{lem} 

Using $S^\K$ and Lemma \ref{LemDivS}, we define the electric single layer potential for $\mathbf{j}\in TH^{-1/2}_{\rm div}({\partial X})$ by
\begin{align}
&\mathcal{S}_{E}^\K\left[\,\mathbf{j}\,\right](\bx):=  S^\K[\,\mathbf{j}\,](\bx)+\frac{1}{\K^2}\nabla S^\K\left[\,{\rm div}_{\partial X}\mathbf{j}\,\right](\bx)= 
-\frac{1}{\epsilon_0} \ds\int_{{\partial X}} \mathbf{\Gamma}(\by,\bx)\,\mathbf{j}(\by)d\sigma(\by),
\end{align}
Moreover, for $\bpsi\in TH^{-1/2}_{\rm div}({\partial X})$, we define the magnetic dipole operator $\mathcal{P}^\K$ by 
\begin{eqnarray}
\mathcal{P}^{\K}[\bpsi](\bx):=\int_{{\partial X}}\nabla_\bx\times \left(g(\bx,\by)\bpsi(\by)\right)d\sigma(\by)\times \nu(\bx), \quad \bx\in{\partial X},
\end{eqnarray} 
and the operator $\mathcal{P}^\K_*$  by  
\begin{eqnarray}
\mathcal{P}^\K_*[\bpsi](\bx):= \mathcal{P}^{\K}[\bpsi\times \nu](\bx)\times\nu(\bx).
\end{eqnarray}

Then, following results hold. 
\begin{lem}[See {\cite[Section 5.5]{nedelec}} and {\cite[Section 6.3]{colton}}]\label{Lem2.4}
The operator $\mathcal{P}^\K$ is continuous mapping from $TH^{-1/2}_{\rm div}({\partial X})$ to itself. The operator $\frac{1}{2}\mathcal{I}-\mathcal{P}^\K$ is Fredholm of index zero from $TH^{-1/2}_{\rm div}({\partial X})$ to itself, where  $\mathcal{I}:TH^{-1/2}_{\rm div}({\partial X})\to TH^{-1/2}_{\rm div}({\partial X})$ is the identity operator. Moreover, for all $\bpsi\in TH^{-1/2}_{\rm div}({\partial X})$
\begin{eqnarray}
\mathcal{P}_*^\K[\bpsi]\times \nu = -\mathcal{P}^\K[\bpsi\times \nu].
\end{eqnarray}
\end{lem}

\begin{lem}[See  {\cite[Lemma 2.6]{CL}}]\label{lemS}
The electric single layer potential $\mathcal{S}^\K_E$ is continuous from $TH^{-1/2}_{\rm div}({\partial X})$ to $H_{\rm loc}(\rm{\bf{curl}},\RR^3)$ and for all $\mathbf{j}\in TH^{-1/2}_{\rm div}({\partial X})$,  
\begin{eqnarray}
\left(\nabla\times\nabla\times-\K^2\mathcal{I}\right)\mathcal{S}^\K_E[\,\mathbf{j}\,](\bx) =\mathbf{0}.
\end{eqnarray} 
Moreover, $\mathcal{S}^\K_E$ satisfies the \emph{Silver-M\"uller} condition.
\end{lem}

\begin{lem}[See {\cite[Theorem 6.12]{colton}}, {\cite[Theorem 5.5.1]{nedelec}} and {\cite[Section 5]{H}}]\label{lemJump}
For all $\mathbf{j}\in TH_{\rm div}^{-1/2}({\partial X})$, the traces $(S^\K_E[\,\mathbf{j}\,]\times\nu)^{\pm}$ and $( (\nabla\times S^\K_E[\,\mathbf{j}\,])\times\nu)^{\pm}$ are well defined and
\begin{eqnarray}
&&\Big(S^\K_E[\,\mathbf{j}\,]\times\nu\Big)^{\pm}
=\mathbf{0},
\\
&&
\Big( (\nabla\times S^\K_E[\,\mathbf{j}\,])\times\nu\Big)^{\pm}=\left(\mp \frac{1}{2}\mathcal{I}+\mathcal{P}^\K\right)[\,\mathbf{j}\,]. 
\end{eqnarray}
Here superscripts $+$ and $-$ indicate the limiting values at ${\partial X}$ from outside and inside $X$ respectively.
\end{lem}


\subsection{Polarization tensor}

Let us define the piecewise constant function $\gamma$ by 
\begin{eqnarray}
\gamma(\bx):=
\begin{cases}
\gamma_0, & \bx\in\RR^3\setminus\overline{X}, 
\\
\gamma_X, & \bx\in X, 
\end{cases}
\end{eqnarray} 
where $\gamma_0,\gamma_X\in\mathbb{C}$ such that $\Re e\{\gamma_0\}, \Re e\{\gamma_X\}>0$, and  let $v_i$ be the scalar potential defined as the solution to the transmission problem 
\begin{eqnarray}
\begin{cases}
\Delta v_i= 0, & (\RR^3 \setminus\overline{X})\cup X, 
\\
v_i^+-v_i^-=0, & {\partial X}, 
\\
\ds\frac{\gamma_0}{\gamma_1}\left(\frac{\partial v_i}{\partial\nu}\right)^+ -\left(\frac{\partial v_i}{\partial \nu}\right)^-=0, & {\partial X},
\\
v_i(\bx)-x_i\to 0, & |\bx|\to \infty. 
\end{cases}\label{vi}
\end{eqnarray}
We define the polarization tensor $\mathbf{M}_X(k):= (m_{ij})_{i,j=1}^3$,  associated with the domain $X$ depending on the contrast $k:=\gamma_0/\gamma_X$, by 
\begin{eqnarray}
m_{ij}\left(k\right):=\ds\frac{1}{k}\int_{X}\frac{\partial v_i}{\partial x_j}d\bx.
\end{eqnarray} 

\begin{lem}[See {\cite[Section 3.1]{AK}}]\label{PT}
The tensor $\mathbf{M}_X(k)$ is real symmetric positive definite if $k\in\RR_+$. Moreover, when $X$ is a ball 
\begin{eqnarray}
\mathbf{M}_X(k)= \dfrac{3}{2k+1}|X|\mathbf{I}_3.
\end{eqnarray}
\end{lem}

\section{Topological derivative based imaging framework} \label{framework}

Let $\bz_S\in\Omega$ be a search point. Nucleate a trial inclusion $D_\delta=\delta B_S+\mathbf{z}_S$ inside the background $\Omega$ with permittivity $\epsilon_\delta$ and permeability $\mu_\delta$ defined by 
\begin{eqnarray}
\mu_\delta(\bx):=
\begin{cases}
\mu_2, & \bx\in D_\delta,
\\
\mu_0, & \bx\in\RR^3\setminus\overline{D_\delta},
\end{cases}
\quad\text{and}\quad
\epsilon_\delta(\bx):=
\begin{cases}
\epsilon_2, & \bx\in D_\delta,
\\
\epsilon_0, & \bx\in\RR^3\setminus\overline{D_\delta},
\end{cases}
\end{eqnarray}
where $\epsilon_2,\mu_2>0$.
Let $\mathbf{H}_\delta$  be the magnetic field in the presence of  inclusion $D_\delta$ in $\Omega$ satisfying a transmission problem analogous to that in \eqref{Hrho}. We collect $\mathbf{H}_{\delta}\times\nu(\bx)$ for all $\bx\in\partial\Omega$ and define the discrepancy functional 
\begin{eqnarray}
\ds\mathcal{H}_f[\mathbf{H}_0](\bz_S)
:= \frac{1}{2}\int_{\partial\Omega} \left|\left(\frac{1}{2}\mathcal{I}-\mathcal{P}^\K\right) \Big[\left(\mathbf{H}_{\rho}-\mathbf{H}_{\delta}\right)\times\nu\Big](\bx)\right|^2 d\sigma(\bx).\label{Dis}
\end{eqnarray}

Here the subscript $f$ substantiates the use of a filter $\left(\frac{1}{2}\mathcal{I}-\mathcal{P}^\K\right)$ in the cost functional. As already established by \citet{AGJK, TDelastic} for the case of Helmholtz and elasticity equations, the identification of the exact location of true inclusion using the classical $L^2-$cost functional over a bounded domain cannot be guaranteed, and the post-processing of the data is necessary. We establish later on that guaranteed identification can be achieved using filtered discrepancy functional $\mathcal{H}_f$. It is emphasized that the post-processing compensates for the effects of an imposed Neumann boundary condition on the magnetic field. 

By construction, the search point $\bz_S$ relative to which the field  $\mathbf{H}_\delta$ minimizes the functional $\mathcal{H}_f[\mathbf{H}_0]$ is a potential candidate for $\bz_D$. 
In order to study the optimization problem \eqref{Dis}, we define the \emph{topological derivative} of misfit $\mathcal{H}_f$ as follows. 
\begin{defn}[Topological derivative] 
For any $\bz_S\in\Omega$ and incident field $\mathbf{H}_0$, the topological derivative (imaging functional) of the misfit $\mathcal{H}_f[\mathbf{H}_0]$, hereafter denoted by ${\partial}_T\mathcal{H}_f[\mathbf{H}_0]$, is defined by 
\begin{eqnarray}
{\partial}_T\mathcal{H}_f[\mathbf{H}_0](\bz_S):= -\ds\frac{\partial\mathcal{H}_f[\mathbf{H_0}](\bz_S)}{\partial(\delta^3)}.
\end{eqnarray}
\end{defn}

\noindent{\textbf{Nota Bene:}}\quad In the sequel, we systematically adopt the following notation for brevity.
$$
\frac{\mu_0}{\mu_1}:=\mu_{1r},
\quad
\frac{\epsilon_0}{\epsilon_1}:=\epsilon_{1r},
\quad
\frac{\mu_0}{\mu_2}:=\mu_{2r},
\quad
\frac{\epsilon_0}{\epsilon_2}:=\epsilon_{2r},
\quad
a_\mu:= (\mu_{2r}-1),
\quad
a_\epsilon:=(\epsilon_{2r}-1),
$$
$$
\mathbf{M}_D^\mu:= \mathbf{M}_{B_D}\left(\mu_{1r}\right),
\quad
\mathbf{M}_D^\epsilon= \mathbf{M}_{B_D}\left(\epsilon_{1r}\right)
\quad
\mathbf{M}_S^\mu:= \mathbf{M}_{B_S}\left(\mu_{2r}\right),
\quad
\mathbf{M}_S^\epsilon= \mathbf{M}_{B_S}\left(\epsilon_{2r}\right)
$$

The following asymptotic expansion of the scattered magnetic field due to the presence of inclusion $D_\rho$ versus scale factor $\rho$, is the key ingredient to evaluate   $\partial_T\mathcal{H}_f[\mathbf{H}_0]$. 

\begin{thm}[See {\cite[Theorem 1]{AVV}}]\label{Asymp}
For all $\bx\in \partial \Omega$, and $D_\rho=\rho B_D+\bz_D$ satisfying \eqref{assumption}  
\begin{align}
\nonumber
\ds\Big(\frac{1}{2}\mathcal{I}-&\mathcal{P}^\K\Big)
\big[(\mathbf{H}_\rho-\mathbf{H_0})\times\nu\big](\bx)
={\rho^3\K^2\left(\mu_{1r}-1\right)}\epsilon_0^{-1}\big[\mathbf{\Gamma}(\bz_D, \bx)\times\nu(\bx)\big]\mathbf{M}_{D}^\mu\mathbf{H}_0(\bz_D)
\nonumber
\\
&\qquad +{\rho^3(\epsilon_{1r}-1)}{\epsilon_0^{-1}}\Big[\Big(\nabla_{\bz_D}\times\mathbf{\Gamma}(\bx,\bz_D)\Big)^T\times\nu(\bx)\Big]\mathbf{M}_{D}^\epsilon\nabla\times\mathbf{H}_0(\bz_D)+O(\rho^4),\label{AE1}
\end{align}
where the term $O(\rho^4)$ is bounded by $C\rho^4$ uniformly on $\bx$ with constant $C$ independent on $\bz_D$.
\end{thm}

Remark that, from Theorem \ref{Asymp}, we also have for all $\bx\in\partial\Omega$ and $\bz_S\in\Omega$
\begin{align}
\nonumber
\ds\Big(\frac{1}{2}\mathcal{I}-\mathcal{P}^\K\Big)
\big[(\mathbf{H}_\delta-&\mathbf{H_0})
\times\nu\big](\bx)
={\delta^3\K^2a_\mu}{\epsilon_0^{-1}}\big[\mathbf{\Gamma}(\bz_S, \bx)\times \nu(\bx)\big] \mathbf{M}_{S}^\mu\mathbf{H}_0(\bz_S)
\nonumber
\\
& 
+{\delta^3a_\epsilon}{\epsilon_0^{-1}}\Big[\Big(\nabla_{\bz_S}\times
\mathbf{\Gamma}(\bx,\bz_S)\Big)^T\times\nu(\bx)\Big]
\mathbf{M}_{S}^\epsilon\nabla\times\mathbf{H}_0(\bz_S)
+O(\delta^4).\label{AE2}
\end{align}

\begin{thm}
For all  $\bz_S\in\Omega$ and incident fields $\mathbf{H}_0$,  
\begin{align}
\partial_T\mathcal{H}_f[\mathbf{H}_0](\bz_S)
= -\Re e\Big\{\K^2a_\mu
&\bU(\bz_S)\cdot \mathbf{M}_{S}^\mu\mathbf{H}_0(\bz_S) 
+a_\epsilon\nabla\times \bU(\bz_S)\cdot \mathbf{M}_{S}^\epsilon\nabla\times\mathbf{H}_0(\bz_S)\Big\},
\label{TD}
\end{align} 
where the \emph{back-propagator} $\bU$ is defined by 
\begin{eqnarray}
\bU(\bz)&:=&\mathcal{S}^\K_E\left[\nu\times\mathbf{W}\right](\bz)
\quad \text{with}\quad  
\mathbf{W}(\bz):=\overline{\left(\frac{1}{2}\mathcal{I}-\mathcal{P}^\K\right)\left[(\mathbf{H}_\rho-\mathbf{H}_0)\times\nu\right](\bz)}.
\label{U}
\end{eqnarray} 
\end{thm}
\begin{proof}
Note that, by virtue of asymptotic expansion \eqref{AE2}, for all $\bz_S\in\Omega$  
\begin{align}
\ds \mathcal{H}_f&[\mathbf{H}_0](\bz_S)-\frac{1}{2}\int_{\partial\Omega}\Big|\left(\frac{1}{2}\mathcal{I}-\mathcal{P}^\K\right)\left[(\mathbf{H}_\rho-\mathbf{H}_0)\times\nu\right](\bz)\Big|^2d\sigma(\bz)=O\left(\delta^6\right)
\nonumber
\\
&\ds-\Re e\Bigg\{\int_{\partial\Omega} \left(\frac{1}{2}\mathcal{I}-\mathcal{P}^\K\right)\left[(\mathbf{H}_\delta-\mathbf{H}_0)\times\nu\right](\bz)\cdot \overline{\left(\frac{1}{2}\mathcal{I}-\mathcal{P}^\K\right)\left[(\mathbf{H}_\rho-\mathbf{H}_0)\times\nu\right](\bz)} d\sigma(\bz)\Bigg\},
\nonumber
\\
=&-\frac{\delta^3a_\epsilon}{\epsilon_0}\Re e\left\{\int_{\partial\Omega}
\left[\left(\nabla_{\bz_S}\times\mathbf{\Gamma}(\bz,\bz_S)\right)^T \times \nu(\bz)\right]\mathbf{M}_{S}^\epsilon
\nabla\times\mathbf{H}_0(\bz_S)\cdot \mathbf{W}(\bz)d\sigma(\bz)\right\}
\nonumber
\\
&-\ds\frac{\delta^3\K^2a_\mu}{\epsilon_0}\Re e\left\{
\int_{\partial\Omega} \mathbf{\Gamma}(\bz_S,\bz)\times\nu(\bz) \mathbf{M}_{S}^\mu\mathbf{H}_0(\bz_S)\cdot \mathbf{W}(\bz) d\sigma(\bz)\right\} + O(\delta^4), \label{expansion}
\end{align}
since $\left(\frac{1}{2}\mathcal{I}-\mathcal{P}^\K\right)[(\mathbf{H}_\delta-\mathbf{H}_0)\times \nu] =O(\delta^3)$ by Theorem \ref{Asymp}. 

Recall that for any matrix $A$, and vectors $\bu$, $\bv$ and $\bw$
$$
\left[(A\times \bu)\bv\right]\cdot\bw= \left[(A\times\bu)^T\bw\right]\cdot\bv= [A^T(\bu\times\bw)]\cdot \bv.
$$
Therefore, we have 
\begin{align*}
\left[\mathbf{\Gamma}(\bz_S,\bz)\times \nu(\bz)\right]
\mathbf{M}_{S}^\mu\mathbf{H}_0(\bz_S)\cdot \mathbf{W}(z)
=
\left[\left(\mathbf{\Gamma}(\bz_S,\bz)\right)^T\nu(\bz)\times\mathbf{W}(\bz) \right]\cdot\mathbf{M}_{S}^\mu\mathbf{H}_0(\bz_S).
\end{align*}
Moreover, from property \eqref{Pro1} and symmetry of $\mathbf{\Gamma}$,
\begin{align}
\int_{\partial\Omega}\mathbf{\Gamma}
&(\bz_S,\bz)\times\nu(\bz) \mathbf{M}_{S}^\mu\mathbf{H}_0(\bz_S)
\cdot \mathbf{W}(\bz)d\sigma(\bz)
\nonumber
\\
&=
\ds\int_{\partial\Omega}\mathbf{\Gamma}(\bz,\bz_S)\big(\nu(\bz)\times\mathbf{W}(\bz)\big)d \sigma(\bz)
\cdot \mathbf{M}_{S}^\mu\mathbf{H}_0(\bz_S)
=
-\epsilon_0\bU(\bz_S)\cdot \mathbf{M}_{S}^\mu\mathbf{H}_0(\bz_S).\label{I1}
\end{align}
Similarly, 
\begin{align}
\int_{\partial\Omega}
&
\left[\big(\nabla_{\bz_S}\times\mathbf{\Gamma}(\bz,\bz_S)\big)^T \times \nu(\bz)\right]
\mathbf{M}_{S}^\epsilon\nabla\times\mathbf{H}_0(\bz_S)
\cdot \mathbf{W}(\bz)d\sigma(\bz),
\nonumber
\\
&=\int_{\partial\Omega}\left[\big(\nabla_{\bz_S}
\times\mathbf{\Gamma} (\bz,\bz_S)\big)^T \times \nu(\bz)\right]^T\mathbf{W}(\bz)d\sigma(\bz)\cdot
\mathbf{M}_{S}^\epsilon\nabla\times\mathbf{H}_0(\bz_S),
\nonumber
\\
&=\int_{\partial\Omega}\nabla_{\bz_S}\times\mathbf{\Gamma}(\bz,\bz_S)\big[\nu(\bz)\times\mathbf{W}(\bz)\big]d\sigma(\bz)\cdot
\big[\mathbf{M}_{S}^\epsilon\nabla\times\mathbf{H}_0(\bz_S)\big],
\nonumber
\\
&=-\epsilon_0 \nabla\times\bU(\bz_S)\cdot
\mathbf{M}_{S}^\epsilon\nabla\times\mathbf{H}_0(\bz_S).\label{I2}
\end{align}

By virtue of \eqref{I1} and \eqref{I2}, expansion \eqref{expansion} renders
\begin{align}
\mathcal{H}_f[\mathbf{H}_0]&(\bz_S)-\frac{1}{2}\int_{\partial\Omega}\Big|\Big(\frac{1}{2}\mathcal{I}-\mathcal{P}^\K\Big)
\big[(\mathbf{H}_\rho-\mathbf{H}_0)\times\nu\big](\bz)\Big|^2d\sigma(\bz)
\nonumber
\\
=& \delta^3\Re e\Big\{\K^2a_\mu \bU(\bz_S)\cdot \mathbf{M}_{S}^\mu\mathbf{H}_0(\bz_S) 
+a_\epsilon\nabla\times \bU(\bz_S)\cdot \mathbf{M}_{S}^\epsilon\nabla\times\mathbf{H}_0(\bz_S)\Big\}
+O(\delta^4).\label{expansion2}
\end{align}
Finally, the conclusion follows by tending $\delta^3\to 0$. 
\end{proof}

To conclude this section, we precise that thanks to Lemma \ref{lemS} and Lemma \ref{lemJump} the back-propagator $\mathbf{U}$ is the solution to boundary value problem 
\begin{eqnarray}
\begin{cases}
\ds \nabla\times\nabla\times\mathbf{U}(\bx)-\K^2\mathbf{U}(\bx) = \mathbf{0}, & \bx\in\Omega, 
\\
\ds\nabla\times\mathbf{U}(\bx)\times\nu(\bx)= \left(\frac{1}{2}\mathcal{I}+\mathcal{P}^\K\right)\overline{\left[\nu\times\left(\frac{1}{2}\mathcal{I}-\mathcal{P}^\K\right)[(\mathbf{H}_\rho-\mathbf{H}_0)\times\nu]\right]}(\bx), & \bx\in\partial\Omega.
\end{cases}
\end{eqnarray}

\section{Sensitivity and resolution analysis} \label{sensitivity}

In this section, we explain why should the topological derivative functional attain its maximum at the true location $\bz_D$ of the electromagnetic inclusion $D$.  

\subsection{Imaging with single incident field}
In order to ascertain the localization and resolution of the imaging function $\partial_T\mathcal{H}_f[\mathbf{H}_0]$, we entertain two special cases for simplicity. Precisely, we consider the dielectric inclusions ($\mu_0=\mu_1$ but $\epsilon_0\neq \epsilon_1$) and permeable inclusions ($\mu_0\neq \mu_1$ but $\epsilon_0=\epsilon_1$) only. The general case  ($\mu_0\neq\mu_1$ and $\epsilon_0\neq \epsilon_1$) can be dealt with analogously, and the same conclusions hold but the analysis is more involved.

Consider the case of a permeable inclusion. Let $\epsilon_2=\epsilon_0$  thereby restricting $\partial_T\mathcal{H}_f[\mathbf{H}_0]$  to 
\begin{align*}
\partial_T\mathcal{H}_f[\mathbf{H}_0](\bz_S)
=& -\K^2a_\mu\Re e\left\{ \bU(\bz_S)\cdot \mathbf{M}_{S}^\mu \mathbf{H}_0(\bz_S)\right\},
\nonumber
\\
=& {\K^2\epsilon_0^{-1}a_\mu}\Re e\left\{\int_{\partial\Omega}\mathbf{\Gamma}(\bz,\bz_S)\left(\nu(\bz)\times\mathbf{W}(\bz)\right)d\sigma(\bz)\cdot \mathbf{M}_{S}^\mu \mathbf{H}_0(\bz_S)\right\}.
\end{align*}

Note that 
\begin{align*}
\mathbf{W}(\bz)
=&\overline{\left(\frac{1}{2}\mathcal{I}-\mathcal{P}^\K\right)
\left[(\mathbf{H}_\rho-\mathbf{H}_0)\times\nu\right](\bz)},
\\
=&
\rho^3{\K^2\epsilon_0^{-1}(\mu_{1r}-1)} \left[\overline{\mathbf{\Gamma}(\bz_D, \bz)}\times\nu(\bz)\right]\mathbf{M}_{D}^\mu\overline{\mathbf{H}_0(\bz_D)}+O(\rho^4).
\end{align*}

Therefore, on injecting back the expression for $\mathbf{W}$, we obtain
\begin{equation}
\partial_T\mathcal{H}_f[\mathbf{H}_0](\bz_S):= \rho^3\K^4\epsilon_0^{-2}C_\mu \Re e\left\{ \mathbf{M}_{S}^\mu\mathbf{H}_0(\bz_S)\cdot \overline{\mathcal{R}_1(\bz_S,\bz_D)}\mathbf{M}_{D}^\mu \overline{\mathbf{H}_0(\bz_D)}\right\}
+O(\rho^4),
\label{TDmu}
\end{equation}
where 
\begin{eqnarray}
C_\mu &:=&(\mu_{1r}-1)(\mu_{2r}-1).\label{Cmu}
\\
\label{R1}
\mathcal{R}_1(\bx,\by)&:=& \int_{\partial\Omega}\Big(\overline{\mathbf{\Gamma}(\bx,\bz)\times\nu(\bz)}\Big)^T\Big(\mathbf{\Gamma}(\by,\bz)\times\nu(\bz)\Big)d\sigma(\bz).
\end{eqnarray}

Recall from Lemma \ref{HKI-lem2}, that for all $\bx,\by\in \Omega$ far from the boundary $\partial\Omega$,  we have  
\begin{align}
\mathcal{R}_1(\bx,\by)\simeq & -{\epsilon_0}\K^{-1}\Im m\big\{\mathbf{\Gamma}(\bx,\by)\big\} 
\end{align}
Therefore, by substituting back the approximation of $\mathcal{R}_1$, we arrive at
\begin{align}
\partial_T\mathcal{H}_f[\mathbf{H}_0](\bz_S)
\simeq -{\rho^3\K^3\epsilon_0^{-1}C_\mu}\Re e
\left\{
\mathbf{M}_{S}^\mu\mathbf{H}_0(\bz_S)\cdot
\Im m\big\{\mathbf{\Gamma}(\bz_S,\bz_D)\big\}
\mathbf{M}_{D}^\mu\overline{\mathbf{H}_0(\bz_D)}
\right\}
+O(\rho^4).
\end{align}

On the other hand, if the inclusion is dielectric, that is, $\mu_1=\mu_0$ and we let $\mu_2=\mu_0$, the topological derivative reduces to 
\begin{align*}
\partial_T\mathcal{H}_f[\mathbf{H}_0](\bz_S)
=&-a_\epsilon\Re e
\left\{ 
\nabla\times \bU(\bz_S)\cdot \mathbf{M}_{S}^\epsilon
\nabla\times\mathbf{H}_0(\bz_S)
\right\},
\\
=&
\epsilon_0^{-1}a_\epsilon\Re e
\left\{
 \nabla\times \int_{\partial\Omega}\mathbf{\Gamma}(\bz,\bz_S)(\nu(\bz)\times\mathbf{W}(\bz)d\sigma(\bz)
\cdot \mathbf{M}_{S}^\epsilon \nabla\times\mathbf{H}_0(\bz_S)
\right\}.
\end{align*} 

In this case,  $\mathbf{W}$ admits the expansion
\begin{align*}
\mathbf{W}(\bz)
=
\rho^3\epsilon_0^{-1}(\epsilon_{1r}-1)
\left[\overline{\big(\nabla_{\bz_D}\times\mathbf{\Gamma}(\bz,\bz_D)\big)^T}\times\nu(\bz)\right]
\mathbf{M}_{D}^\epsilon\overline{\nabla\times\mathbf{H}_0(\bz_D)}+O(\rho^4).
\end{align*}
Therefore, $\partial_T\mathcal{H}_f[\mathbf{H}_0]$ becomes 
\begin{align}
\partial_T\mathcal{H}_f[\mathbf{H}_0](\bz_S)
=& \rho^3\epsilon_0^{-2}C_\epsilon\Re e\Bigg\{\int_{\partial\Omega} \nabla_{\bz_S}\times \mathbf{\Gamma}(\bz,\bz_S)\Big[\nu(\bz)\times\Big(\overline{
\left(\nabla_{\bz_D}\times\mathbf{\Gamma}(\bz,\bz_D)\right)^T}\times\nu(\bz)
\nonumber
\\
&
\qquad\qquad
\mathbf{M}_{D}^\epsilon\overline{\nabla\times
\mathbf{H}_0(\bz_D)}\Big)\Big]d\sigma(\bz)\cdot
\mathbf{M}_{S}^\epsilon\nabla\times\mathbf{H}_0(\bz_S)\Bigg\}+ O(\rho^4),
\nonumber
\\
=& \rho^3\epsilon_0^{-2}C_\epsilon\Re e\Bigg\{\int_{\partial\Omega} \left(\nabla_{\bz}\times \mathbf{\Gamma}(\bz_S,\bz)\right)^T
\Big[\nu(\bz)\times\Big(\overline{
\left(\nabla_{\bz}\times\mathbf{\Gamma}(\bz_D,\bz)\right)}\times\nu(\bz)
\nonumber
\\
&
\qquad\qquad\mathbf{M}_{D}^\epsilon\overline{\nabla
\times\mathbf{H}_0(\bz_D)}\Big)\Big]d\sigma(\bz)\cdot
\mathbf{M}_{S}^\epsilon\nabla\times\mathbf{H}_0(\bz_S)\Bigg\}+ O(\rho^4),
\nonumber
\\
=&\rho^3\epsilon_0^{-2}C_\epsilon \Re e\left\{ \mathbf{M}_{S}^\epsilon\nabla \times\mathbf{H}_0(\bz_S) \cdot
\overline{\mathcal{R}_2(\bz_S,\bz_D)} \mathbf{M}_{D}^\epsilon
\overline{\nabla\times\mathbf{H}_0(\bz_D)}\right\} 
+ O(\rho^4),\label{TDepsilon}
\end{align}
where   
\begin{eqnarray}
C_\epsilon&:=&(\epsilon_{1r}-1)(\epsilon_{2r}-1).\label{Cepsilon}
\\
\label{R2}
\mathcal{R}_2(\bx,\by)&:=& \int_{\partial\Omega}\Big(\overline{\nabla_\bz\times \mathbf{\Gamma}(\bx,\bz)\times\nu(\bz)}\Big)^T\Big(\nabla_\bz\times\mathbf{\Gamma}(\by,\bz)\times\nu(\bz)\Big)d\sigma(\bz).
\end{eqnarray}

Note that, by virtue of the assumption \eqref{assumption}, and the Silver-M\"uller condition, for all $\bx,\by\in\Omega$ away from boundary $\partial\Omega$
\begin{align}
\mathcal{R}_2(\bx,\by)\simeq& \K^2 \int_{\partial\Omega}\Big(\overline{\mathbf{\Gamma}(\bx,\bz)}\Big)^T\Big(\mathbf{\Gamma}(\by,\bz)\Big)d\sigma(\bz)
\simeq 
-{\K}{\epsilon_0}\Im m\Big\{\mathbf{\Gamma}(\bx,\by)\Big\},
\end{align}
where for the latter identity, Lemma \ref{HKI-lem1} is invoked.
Therefore, we conclude that 
\begin{align}
\partial_T\mathcal{H}_f[\mathbf{H}_0](\bz_S)
\simeq-\frac{\rho^3\K C_\epsilon}{\epsilon_0}\Re e\left\{\mathbf{M}_{S}^\epsilon\nabla \times\mathbf{H}_0(\bz_S) \cdot
\Im m\big\{\mathbf{\Gamma}(\bz_S,\bz_D)\big\} \mathbf{M}_{D}^\epsilon\overline{\nabla \times\mathbf{H}_0(\bz_D)} \right\}+ O(\rho^4).
\end{align}

\subsubsection{Sign and decay properties of topological derivative}

For both dielectric and permeable inclusions, 
\begin{eqnarray}
\partial_T\mathcal{H}_f[\mathbf{H}_0](\bz_S)\propto \Im m\big\{\mathbf{\Gamma}(\bz_S,\bz_D)\big\}= -\frac{\epsilon_0\K}{4\pi}\left[\frac{2}{3}j_0(\K r)\mathbf{I}_3+j_2(\K r)\left(\widehat{\mathbf{r}}\,\widehat{\mathbf{r}}^T-\frac{1}{3}\mathbf{I}_3\right)\right],\label{asymptoticTD}
\end{eqnarray}
where $j_0$ and $j_2$ are the spherical Bessel functions of first kind and $\mathbf{r}:=\bz_S-\bz_D$ with $r:=|\mathbf{r}|$ and $\hat{\mathbf{r}}=\mathbf{r}/r$. Since $j_n(kr)=O(1/kr)$ as $kr\to \infty$ (see, for instance \cite[10.52.3]{NIST}), the functional $\bz_S\to \partial_T\mathcal{H}_f[\mathbf{H}_0](\bz_S)$ rapidly decays for $\bz_S$ away from $\bz_D$ and has a sharp peak when $\bz_S\to\bz_D$ with a focal spot size of half a wavelength of the incident wave. Therefore, the resolution of the imaging functional $\bz_S\to\partial_T\mathcal{H}_f[\mathbf{H}_0](\bz_S)$  achieves the Rayleigh resolution limit. 
Moreover, it synthesizes the sensitivity of $\mathcal{H}_f[\mathbf{H}_0](\bz_S)$ relative to the insertion of an inclusion at the search location $\bz_S\in\Omega$. Heuristically, if the contrasts $(\epsilon_{2r}-1)$ and $(\mu_{2r}-1)$ have the same signs as $(\epsilon_{1r}-1)$ and $(\mu_{1r}-1)$ respectively, then the functional $\mathcal{H}_f[\mathbf{H}_0](\bz_S)$ must observe the most pronounced decrease at the potential candidate $\bz_S\in\Omega$ for the true location $\bz_D$. In other words, $\bz_S\to \frac{\partial\mathcal{H}_f[\mathbf{H}_0]}{\partial(\delta^3)}(\bz_S)$ is expected to attain its most pronounced negative value; refer, for instance, to \cite{BG, GuzinaChik, Bellis} for detailed discussions on sign heuristic.

Notice that both $C_\mu$ and $C_\epsilon$ are positive if the contrasts of true and trial inclusions have same signs. Consequently, by virtue of the decay property, $\partial_T\mathcal{H}_f[\mathbf{H}_0](\bz_S)=-\frac{\partial\mathcal{H}_f[\mathbf{H}_0]}{\partial(\delta^3)}(\bz_S)$
assumes its maximum positive value for both dielectric and permeable inclusions when $\bz_S\to\bz_D$. Thus, the functional  $\partial_T\mathcal{H}_f[\mathbf{H}_0](\bz_S)$  possesses a sharpest decay  and the functional $\mathcal{H}_f[\mathbf{H}_0](\bz_S)$ decreases rapidly when the contrasts are chosen to have same signs.

\subsection{Imaging with multiple incident fields} \label{MultiSec}

Let $\theta_1, \theta_2, \cdots ,\theta_n$ be $n-$equidistributed  directions on the unit sphere and let 
\begin{eqnarray}
\mathbf{H}_0^{j,\ell}(\bx)= \theta_j^{\perp,\ell} e^{i\K\theta_j^T\bx}, \qquad j\in\{1,2,\cdots,n\},\quad  \ell\in\{1,2\}\label{Hj}
\end{eqnarray}
be the incident magnetic fields where $\theta^{\perp,\ell}_j$ are the polarization directions such that $(\theta_j,\theta^{\perp,1}_j,\theta^{\perp,2}_j)$ forms an orthonormal basis of $\RR^3$. The incident fields $\mathbf{H}_0^{j,\ell}$ are the solutions to the Maxwell equations 
\begin{eqnarray}
\begin{cases}
\ds\nabla\times\nabla \times \mathbf{H}_0^{j,\ell}-\K^2\mathbf{H}_0^{j,\ell} = \mathbf{0}, & \Omega, 
\\
\ds\mathbf{h}^{j,\ell}(\bx)= {\epsilon_0^{-1}}\nabla\times \left(\theta_j^{\perp,\ell} e^{i\K\theta_j^T\bx}\right)\times \nu(\bx), & \partial\Omega.
\end{cases}
\end{eqnarray}

We recall  that for $n$ sufficiently large 
\begin{eqnarray}
\frac{1}{n}\sum_{j=1}^n e^{i\K\theta_j^T(\bx-\by)}\simeq \frac{4\pi}{\K}\Im m\big\{g(\bx,\by)\big\}.
\end{eqnarray}
Since $(\theta_j, \theta_j^{\perp,1},\theta_j^{\perp,2})$ form a basis of $\RR^3$, therefore  
$ \theta_j^{\perp,1}(\theta_j^{\perp,1})^T+\theta_j^{\perp,2}(\theta_j^{\perp,2})^T=(\mathbf{I}_3-\theta_j\theta_j^T) $
so that 
\begin{align}
\nonumber
\frac{1}{n}\sum_{\ell=1}^2\sum_{j=1}^n  e^{i\K\theta_j^T(\bx-\by)} \theta_j^{\perp,\ell} \left(\theta_j^{\perp,\ell}\right)^T
=& \frac{1}{n}\sum_{j=1}^n \left(\mathbf{I}_3-\theta_j\theta_j^T\right)
 e^{i\K\theta_j^T(\bx-\by)},
\\
\nonumber
= &\frac{1}{n} \sum_{j=1}^n\left(\mathbf{I}_3-\frac{1}{\K^2}\nabla\nabla^T\right) e^{i\K\theta_j^T(\bx-\by)},
\\
\nonumber
= & \left(\mathbf{I}_3-\frac{1}{\K^2}\nabla\nabla^T\right) \frac{1}{n} \sum_{j=1}^n e^{i\K\theta_j^T(\bx-\by)},
\\
=& 
-\frac{4\pi}{\K\epsilon_0}\Im m\big\{\mathbf{\Gamma}(\bx,\by)\big\}.\label{theta1}
\end{align}
Similarly,  
\begin{align}
\frac{1}{n}\sum_{\ell=1}^2\sum_{j=1}^n  e^{i\K\theta_j^T(\bx-\by)} \left(\theta_j\times \theta_j^{\perp,\ell}\right)\left(\theta_j\times \theta_j^{\perp,\ell}\right)^T
\simeq 
-\frac{4\pi}{\K\epsilon_0}\Im m\big\{\mathbf{\Gamma}(\bx,\by)\big\}.\label{theta2}
\end{align}

Let us define the topological derivative for multiple incident fields by  
\begin{equation}\label{MultiFunc}
\ds\partial_T\mathcal{H}_f(\bz_S):= \frac{1}{n}\sum_{\ell=1}^2\sum_{j=1}^n
\partial_T\mathcal{H}_f[\mathbf{H}_0^{j,\ell}](\bz_S).
\end{equation}

The following result holds
\begin{thm}\label{multiResolution}
Let $\bz_S\in\Omega$, $D=\rho B_D +\bz_D$ satisfy \eqref{assumption} and $n\in\mathbb{N}$ be sufficiently large. Then, 
\begin{enumerate}
\item for a permeable inclusion ($\epsilon_0=\epsilon_1=\epsilon_2$)
\begin{eqnarray}
\partial_T\mathcal{H}_f(\bz_S) \simeq 
\frac{4\pi \rho^3\K^2 C_\mu}{\epsilon_0^2}\Re e\Big\{\Im m\big\{\mathbf{\Gamma}(\bz_S,\bz_D)\big\}\mathbf{M}_{D}^\mu:
\mathbf{M}_{S}^\mu\Im m\big\{\mathbf{\Gamma}(\bz_S,\bz_D)\big\}\Big\} +O(\rho^4).\label{TDmulti1}
\end{eqnarray}
\item  for a dielectric inclusion ($\mu_0=\mu_1=\mu_2$)
\begin{eqnarray}
\partial_T\mathcal{H}_f(\bz_S)
\simeq \frac{4\pi\rho^3\K^2 C_\epsilon}{\epsilon_0^2}  \Re e\Big \{ \Im m\big\{\mathbf{\Gamma}(\bz_S,\bz_D)\big\} \mathbf{M}_{D}^\epsilon:
\mathbf{M}_{S}^\epsilon\Im m\big\{\mathbf{\Gamma}(\bz_S,\bz_D)\big\} 
\Big\}
+ O(\rho^4).\label{TDmulti2}
\end{eqnarray}
The constants $C_\mu$ and $C_\epsilon$ are defined by \eqref{Cmu} and \eqref{Cepsilon} respectively.
\end{enumerate}
\end{thm}

\begin{proof}
When $\epsilon_0=\epsilon_1=\epsilon_2$, for all $\bz_S\in\Omega$ 
\begin{align*}
\partial_T&\mathcal{H}_f(\bz_S)
=\frac{1}{n}\sum_{\ell=1}^2\sum_{j=1}^n  \partial_T\mathcal{H}_f[\mathbf{H}_0^{j,\ell}](\bz_S),
\\
= &- \frac{\rho^3\K^3C_\mu}{\epsilon_0 n}\sum_{\ell=1}^2\sum_{j=1}^n \Re e\Big\{ \mathbf{M}_{S}^\mu\mathbf{H}_0^{j,\ell}(\bz_S)\cdot\Im m\big\{\mathbf{\Gamma}(\bz_S,\bz_D)\big\}\mathbf{M}_{D}^\mu\overline{\mathbf{H}_0^{j,\ell}(\bz_D)}\Big\}
+O(\rho^4),
\\
= &- \frac{\rho^3\K^3C_\mu}{\epsilon_0}\Re e\Big\{\Im m\big\{\mathbf{\Gamma}(\bz_S,\bz_D)\big\}{\mathbf{M}_{D}^\mu}:\mathbf{M}_{S}^\mu\Big[\frac{1}{n}\sum_{\ell=1}^2\sum_{j=1}^n\theta_j^{\perp,\ell}\left(\theta_j^{\perp,\ell}\right)^T e^{i\K\theta_j^T(\bz_S-\bz_D)}\Big]
\Big\}
+O(\rho^4).
\end{align*}
Here we have made use of the fact that $\theta_j\cdot\mathbf{A}\theta_j= \theta_j\theta_j^T:\mathbf{A}$. 
Finally, \eqref{TDmulti1} follows immediately by virtue of \eqref{theta1}.

In order to prove the other identity, we proceed in the similar fashion. Consider
\begin{align*}
\partial_T&\mathcal{H}_f(\bz_S)
=\frac{1}{n}\sum_{\ell=1}^2\sum_{j=1}^n  \partial_T\mathcal{H}_f[\mathbf{H}_0^{j,\ell}](\bz_S),
\\
=& -\frac{\rho^3\K C_\epsilon}{\epsilon_0 n}\sum_{\ell=1}^2\sum_{j=1}^n\Re e
\Big\{ 
\mathbf{M}_{S}^\epsilon \nabla \times\mathbf{H}^{j,\ell}_0(\bz_S)\cdot 
\Im m\big\{\mathbf{\Gamma}(\bz_S,\bz_D)\big\} \mathbf{M}_{D}^\epsilon\overline{\nabla\times
\mathbf{H}_0^{j,\ell}(\bz_D)}
\Big\}
+ O(\rho^4),
\\
=& - \frac{\rho^3\K^3C_\epsilon}{\epsilon_0 n}\sum_{\ell=1}^2\sum_{j=1}^n\Re e
\Big\{ \mathbf{M}_{S}^\epsilon\left(\theta_j\times\theta_j^{\perp,\ell}\right)\cdot 
\Im m\big\{\mathbf{\Gamma}(\bz_S,\bz_D)\big\} \mathbf{M}_{D}^\epsilon\left(\theta_j\times\theta_j^{\perp,\ell}\right) e^{i\K\theta_j\cdot(\bz_S-\bz_D)} 
\Big\}
+ O(\rho^4),
\end{align*}

On further simplification, we arrive at 
\begin{align*}
\partial_T\mathcal{H}_f(\bz_S)
=& - \frac{\rho^3\K^3 C_\epsilon}{\epsilon_0}\Re e
\Big\{ 
\Im m\big\{\mathbf{\Gamma}(\bz_S,\bz_D)\big\}{\mathbf{M}_{D}^\epsilon}:\mathbf{M}_{S}^\epsilon
\\
&
\qquad \qquad
\frac{1}{n}\sum_{\ell=1}^2\sum_{j=1}^n \left(\theta_j\times\theta_j^{\perp,\ell} \right)\left(\theta_j\times\theta_j^{\perp,\ell}\right)^T
e^{i\K\theta_j\cdot(\bz_S-\bz_D)} 
\Big\}
+ O(\rho^4),
\end{align*}
The proof is completed by invoking approximation \eqref{theta2}. 
\end{proof}

As an immediate consequence of  Theorem \ref{multiResolution} and Lemma \ref{PT}, the following result can be readily proved.

\begin{cor}\label{cor}
Let $\bz_S\in\Omega$, $D=\rho B_D +\bz_D$ be an open sphere in $\RR^3$ such that  condition \eqref{assumption} holds and $n\in\mathbb{N}$ be sufficiently large. Then, 
\begin{enumerate}
\item for a permeable inclusion ($\epsilon_0=\epsilon_1=\epsilon_2$)
\begin{align}
\partial_T\mathcal{H}_f(\bz_S) \simeq & \rho^3\K^2\widetilde{C}_\mu\left\|\Im m\big\{\mathbf{\Gamma}(\bz_S,\bz_D)\big\}\right\|^2 +O(\rho^4).
\end{align}
\item  for a dielectric inclusion ($\mu_0=\mu_1=\mu_2$) 
\begin{align}
\partial_T\mathcal{H}_f(\bz_S) \simeq & \rho^3\K^2\widetilde{C}_\epsilon\left\|\Im m\big\{\mathbf{\Gamma}(\bz_S,\bz_D)\big\}\right\|^2 +O(\rho^4).
\end{align}
The constants $\widetilde{C}_\mu$ and $\widetilde{C}_\epsilon$ are defined by 
\begin{eqnarray}
\widetilde{C}_\mu:= \frac{36\pi\mu_1\mu_2 C_\mu}{\epsilon_0^2(2\mu_0+\mu_1)(2\mu_0+\mu_2)}|B_D|\,|B_S|,
\quad
\widetilde{C}_\epsilon:=\frac{36\pi\epsilon_1\epsilon_2 C_\epsilon}{\epsilon_0^2(2\epsilon_0+\epsilon_1)(2\epsilon_0+\epsilon_2)}|B_D|\,|B_S|.
\end{eqnarray}
\end{enumerate}
\end{cor}

In rest of this paper, we analyze the stability of the multi-incidence imaging functional \eqref{MultiFunc} with respect to medium and measurement noises. 

\section{Statistical stability with respect to measurement noise}\label{measNoise}

The aim here is to substantiate that the imaging functional proposed in Section \ref{MultiSec} is stable with respect to additive measurement noise. For brevity, the simplest model of the measurement noise is entertained. Precisely, it is assumed that the accurate value of magnetic field at the boundary is corrupted by a mean-zero circular Gaussian noise $\etab:\partial\Omega\to\mathbb{C}^3$, with covariance $\sigma^2_{\rm noise}$, that is, 
\begin{eqnarray}
\mathbf{H}_\rho(\bz):= \mathbf{H}^{\rm true}_\rho(\bz)+\etab(\bz), \quad\quad \bz\in\partial\Omega,
\end{eqnarray}
where $\mathbf{H}_\rho$ is the corrupted value of the magnetic field at the boundary.
\newline
\textbf{Nota Bene.}\quad 
In the sequel, $\mathbb{E}$ denotes the expectation with respect to the statistics of the noise. In this section, a superposed \emph{true} indicates the true value of a quantity, that is, the value without noise corruption.

We assume that $\etab$ satisfies following five properties. 
\begin{enumerate}
\item The measurement noises at different locations on the boundary are uncorrelated. 

\item The different components of the measurement noise are uncorrelated. 

\item The real and imaginary parts of the measurement noise are uncorrelated. 

\item The measurement  noises corresponding to two different incident waves are uncorrelated.

\item All the noises corresponding to individual measurements have same variance $\sigma^2_{\rm noise}$.
\end{enumerate}

Then, under aforementioned assumptions, we have
\begin{eqnarray}
\mathbb{E}\left[\etab(\by)\overline{\etab(\by')}^T\right]&=&\sigma_{\rm noise}^2\delta_\by(\by')\mathbf{I}_3,\label{EE}
\\
\mathbb{E}\left[\etab^j(\by)\overline{\etab^{j'}(\by')}^T\right]&=&\sigma_{\rm noise}^2\delta_{jj'}\delta_\by(\by')\mathbf{I}_3,\label{EEcor}
\end{eqnarray}
where superposed $j$ and $j'$ indicate respectively the $j-$th and $j'-$th measurements and $\delta_{jj'}$ is the Kronecker's delta function which assumes the value $1$ when $j=j'$ and zero otherwise.

The imaging functional $\partial_T\mathcal{H}_f(\bz)$ is mainly affected by the additive noise during the back-propagation step due to the construction of back-propagator  in terms of the measurements at the boundary. In the presence of measurement noise, for all $\bz\in\Omega$ the back-propagator takes on the form
\begin{align}
\mathbf{U}(\bz)=& -\frac{1}{\epsilon_0} \int_{\partial\Omega}\mathbf{\Gamma}(\bx,\bz)
\nu(\bx)\times 
\overline{\left(\frac{1}{2}\mathcal{I}-\mathcal{P}^\K\right)\left[\left(\mathbf{H}^{\rm true}_\rho-\mathbf{H}_0+\etab\right)\times\nu\right](\bx)}d\sigma(\bx),
\nonumber
\\
=& \mathbf{U}^{\rm true}(\bz)+ \mathbf{U}^{\rm noise}(\bz),
\end{align}
where $\mathbf{U}^{\rm true}$ corresponds to the back-propagation of the noise-free data whereas $\mathbf{U}^{\rm noise}$ corresponds to noise back-propagation and is given by 
\begin{eqnarray}
\label{Unoise} 
\mathbf{U}^{\rm noise}(\bz):= -\frac{1}{\epsilon_0}\int_{\partial\Omega} \mathbf{\Gamma}(\bx,\bz)\nu(\bx)\times\overline{\left(\frac{1}{2}\mathcal{I}-\mathcal{P}^\K\right)\left[\etab\times\nu\right](\bx)} d\sigma(\bx).
\end{eqnarray}

Let us now discuss the statistics of $\mathbf{U}^{\rm noise}(\bz)$. We have the following lemma. 

\begin{lem} \label{lemUnoise}
The random field $\mathbf{U}^{\rm noise}(\bz)$, $\bz\in\Omega$, is a mean zero Gaussian field with covariance 
\begin{align}
\label{CovUnoise}
\mathbb{E}\Big[\mathbf{U}^{\rm nosise}(\bz)
\overline{\mathbf{U}^{\rm nosise}(\bz')}^T\Big]
\simeq-{\sigma^2_{\rm noise}}{(4\K\epsilon_0)^{-1}}\Im m\big\{\mathbf{\Gamma}(\bz,\bz')\big\}, \qquad \forall \bz,\bz'\in\Omega.
\end{align}
\end{lem} 

The reader is refered to Appendix \ref{App.Proof5.1} for the proof.  Lemma \ref{lemUnoise} indicates that $\mathbf{U}^{\rm noise}$ is a speckle pattern, that is, a random cloud of hot spots having  typical diameters of the order of wavelength and amplitudes of the order of $\sigma_{\rm noise} /(2\sqrt{\K\epsilon_0})$.  

We are now ready to perform the stability analysis of the imaging functional $\partial_T\mathcal{H}_f$. For brevity, we restrict ourselves only to the cases of permeable and dielectric inclusions, however the results extend to the cases otherwise. 

\subsection{Stability analysis for permeable inclusions}\label{sec.5.1}

Recall that the imaging functional for a permeable inclusion reduces to 
\begin{align}
\partial_T&\mathcal{H}_f(\bz)
= -\frac{\K^2a_\mu}{n}\sum_{\ell=1}^2\sum_{j=1}^n \Re e\left\{ \Big(\bU^{{\rm ture},j,\ell}(\bz)+\bU^{{\rm noise},j,\ell}(\bz)\Big)\cdot \mathbf{M}_{S}^\mu\mathbf{H}_0^{j,\ell}(\bz)\right\},\label{TDClutter}
\end{align}
where superposed $j$ and $\ell$ indicate the fields associated with incident wave $\mathbf{H}_0^{j,\ell}$.  It is straight forward that the first term in the above expression with $\mathbf{U}^{{\rm true},j,\ell}$ is identical to the one discussed in Section \ref{MultiSec} and renders the true image obtained in the case without medium noise. The second term introduces a corruption in the image due to the measurement noise. Albeit, the main peak of the true imaging functional is buried in the random cloud of hot spots due to noise, yet it is not altered. Let us compute the covariance of the corrupted image by 
\begin{align*}
&{\rm Cov}\big(\partial_T\mathcal{H}_f(\bz),\partial_T\mathcal{H}_f(\bz')\big)
\\
=
&
\frac{\K^4a_\mu^2}{n^2}\sum_{\ell,\ell'=1}^2\sum_{j,j'=1}^n
\mathbb{E}
\Big[
\Re e\Big\{ \bU^{{\rm noise},j,\ell}(\bz)\cdot \Big[\mathbf{M}_{S}^\mu\mathbf{H}_0^{j,\ell}(\bz)\Big]\Big\}\Re e\Big\{ \bU^{{\rm noise},j',\ell'}(\bz')\cdot \Big[\mathbf{M}_{S}^\mu\mathbf{H}_0^{j',\ell'}(\bz')\Big]\Big\}
\Big],
\\
=&\frac{\K^4a_\mu^2}{2n^2}\sum_{\ell=1}^2\sum_{j=1}^n
\Re e\left\{\Big[\mathbf{M}_{S}^\mu\mathbf{H}_0^{j,\ell}(\bz)\Big]\cdot\mathbb{E}
\Big[
\bU^{{\rm noise},j,\ell}(\bz)\overline{\bU^{{\rm noise},j,\ell}(\bz')}^T\Big]\mathbf{M}_{S}^\mu\overline{\mathbf{H}_0^{j,\ell}(\bz')}
\right\},
\end{align*}
where we have made use of the assumption that $\bU^{{\rm noise},j,\ell} $ and $\bU^{{\rm noise},j',\ell'}$ are uncorrelated.
Using the expression \eqref{Hj} for $\mathbf{H}^{j,\ell}_0$, Lemma \ref{lemUnoise} and the approximation \eqref{theta1}, we obtain 
\begin{align*}
{\rm Cov}&\big(\partial_T\mathcal{H}_f(\bz),\partial_T\mathcal{H}_f(\bz')\big)
\\
=&-\frac{a_\mu^2\sigma^2_{\rm noise}\K^3}{8n^2\epsilon_0}\sum_{\ell=1}^2\sum_{j=1}^n
\Re e\Big\{\Big[\mathbf{M}_{S}^\mu\mathbf{H}_0^{j,\ell}(\bz)\Big]\cdot\Im m\Big\{\mathbf{\Gamma}(\bz,\bz')\Big\}\mathbf{M}_{S}^\mu\overline{\mathbf{H}_0^{j,\ell}(\bz')}
\Big\},
\\
=&-\frac{a_\mu^2\sigma^2_{\rm noise}\K^3}{8n^2\epsilon_0}\sum_{\ell=1}^2\sum_{j=1}^n
\Re e\Big\{\mathbf{M}_{S}^\mu\theta^{\perp,\ell}_j\cdot\Im m\big\{\mathbf{\Gamma}(\bz,\bz')\big\}{\mathbf{M}_{S}^\mu}\theta^{\perp,\ell}_j e^{i\K\theta_j^T(\bz-\bz')}
\Big\},
\\
=&-\frac{a_\mu^2\sigma^2_{\rm noise}\K^3}{8n^2\epsilon_0}
\Re e\Big\{\Im m\big\{\mathbf{\Gamma}(\bz,\bz')\big\}{\mathbf{M}_{S}^\mu}: \mathbf{M}_S^\mu\Big[\frac{1}{n}\sum_{\ell=1}^2\sum_{j=1}^n\theta^{\perp,\ell}_j\left(\theta^{\perp,\ell}_j\right)^T e^{i\K\theta_j^T(\bz-\bz')}\Big]
\Big\},
\\
\simeq &\frac{\pi a_\mu^2\sigma^2_{\rm noise}\K^2}{2n\epsilon_0^2}
\Re e\left\{\Im m\big\{\mathbf{\Gamma}(\bz,\bz')\big\}\mathbf{M}_{S}^\mu: \mathbf{M}^\mu_{S}\Im m\big\{\mathbf{\Gamma}(\bz,\bz')\big\}\right\}.
\end{align*}
This shows that the typical shape of the hot spots created by the additive noise are exactly of the form of the main peak.  Thus the perturbation in the image due to measurement noise is of order $\sigma_{\rm noise}/\sqrt{2n}$ and the typical shape  of hot spots in the perturbation is identical with that of the main peak of functional $\partial_T\mathcal{H}_f$ related to accurate data. The main peak of $\partial_T\mathcal{H}_f$ is not altered by the perturbations. Moreover, since the typical size of the perturbation is inversely proportional to $\sqrt{2n}$, the use of multiple incident fields further enhances the stability of the imaging framework based on $\partial_T\mathcal{H}_f$. 

For a particular case of spherical inclusions, Lemma \ref{PT} yields 
\begin{align}
&{\rm Cov}\big(\partial_T\mathcal{H}_f(\bz),\partial_T\mathcal{H}_f(\bz')\big)
\simeq {\sigma^2_{\rm noise}\widetilde{a}_\mu^2\K^2}(2n)^{-1} 
\big\|\Im m\big\{\mathbf{\Gamma}(\bz,\bz')\big\}
\big\|^2,\label{Cov2} 
\end{align}
where 
$$
\widetilde{a}_\mu 
= \frac{3\sqrt{\pi} \mu_2|B_S|}{\epsilon_0(2\mu_0+\mu_2)}a_\mu=\frac{3\sqrt{\pi}(\mu_0-\mu_2)|B_S|}{\epsilon_0(2\mu_0+\mu_2)}.
$$
It follows immediately from \eqref{Cov2} that the variance of  $\partial_T\mathcal{H}_f$ at $\bz_S$ is given by 
\begin{eqnarray}
\label{var1} 
{\rm Var}\big(\partial_T\mathcal{H}_f(\bz_S)\big)\simeq  {\sigma^2_{\rm noise}\widetilde{a}_\mu^2\K^2}(2n)^{-1}
\big\|\Im m\big\{\mathbf{\Gamma}(\bz,\bz)\big\}
\big\|^2.
\end{eqnarray}
Therefore, the signal-to-noise ratio (SNR), defined by 
\begin{eqnarray}
{\rm SNR}:= \frac{\mathbb{E}\left[\partial_T\mathcal{H}_f(\bz_D)\right]}{\sqrt{{\rm Var}\left[\partial_T\mathcal{H}_f(\bz_D)\right]}},
\end{eqnarray}
can be approximated by virtue of expression \eqref{var1} and Corollary \ref{cor} as
\begin{eqnarray}
{\rm SNR}\simeq \frac{\sqrt{2n}\K\rho^3\widetilde{C}_\mu}{\widetilde{a}_\mu\sigma_{\rm noise}}\Big\|\Im m\big\{\mathbf{\Gamma}(\bz_D,\bz_D)\big\}\Big\|= \frac{12\sqrt{2n\pi}\K\rho^3|B_D||\mu_0-\mu_1|}{\epsilon_0(2\mu_0+\mu_1)\sigma_{\rm noise}}\Big\|\Im m\big\{\mathbf{\Gamma}(\bz_D,\bz_D)\big\}\Big\|.
\end{eqnarray}
Recall that 
$$
\Im m\big\{\mathbf{\Gamma}(\bz_S,\bz_D)\big\}
=-\frac{\epsilon_0\K}{4\pi}\left[
\frac{2}{3}j_0(\K r)\mathbf{I}_3+j_2(\K r)\left(\widehat{\mathbf{r}}\widehat{\mathbf{r}}^T-\frac{1}{3}\mathbf{I}_3\right)
\right].
$$
Therefore, the behavior of $j_0$ and $j_2$ for $r\to 0$ dictates that 
$
\Im m\big\{\mathbf{\Gamma}(\bz_D,\bz_D)\big\}\simeq -{\epsilon_0\K}{(6\pi)^{-1}}\mathbf{I}_3.
$
Consequently,
\begin{eqnarray}
{\rm SNR}\simeq 
\frac{12\sqrt{n}\K^2\rho^3|B_D||\mu_0-\mu_1|}{\sqrt{2\pi}(2\epsilon_0+\epsilon_1)\sigma_{\rm noise}}.
\end{eqnarray}
This elucidates that signal-to-noise ratio depends directly on the volume of the inclusion $\rho^3|B_D|$, the operating wavenumber $\K$ and the contrast $\mu_0-\mu_1$, and inversely proportional to the noise standard deviation $\sigma_{\rm noise}$. 

\subsection{Stability analysis for dielectric inclusions}

For the case when $D$ is a dielectric inclusion, we have 
\begin{align*}
\partial_T&\mathcal{H}_f(\bz)
= -\frac{a_\epsilon}{n}\sum_{\ell=1}^2\sum_{j=1}^n \Re e\Big\{ \nabla\times\Big(\bU^{{\rm ture},j,\ell}(\bz)+\bU^{{\rm noise},j,\ell}(\bz)\Big)\cdot \mathbf{M}_S^\epsilon \nabla\times\mathbf{H}_0^{j,\ell}(\bz')\Big]\Big\}.
\end{align*}
Observe again that the first term corresponds to the true image in the absence of the noise as for the case of permeable inclusions whereas the covariance of corrupted image is now given by 
\begin{align*}
&{\rm Cov}\big(\partial_T\mathcal{H}_f(\bz),\partial_T\mathcal{H}_f(\bz')\big),
\\
=&\frac{a_\epsilon^2}{2n^2}\sum_{\ell,\ell'=1}^2\sum_{j,l=1}^n
\Re e\Big\{\mathbb{E}\Big[
\nabla\times\mathbf{U}^{{\rm noise},j,\ell}(\bz)\cdot\mathbf{M}_S^\epsilon\nabla\times\mathbf{H}^{j,\ell}_0(\bz)
\overline{\nabla\times\mathbf{U}^{{\rm noise},j',\ell'}(\bz')\cdot\mathbf{M}_S^\epsilon\nabla\times\mathbf{H}^{j',\ell'}_0(\bz')}
\Big]\Big\},
\\
=&\frac{a_\epsilon^2}{2n^2}\sum_{\ell=1}^2\sum_{j=1}^n
\Re e\Big\{\mathbf{M}_S^\epsilon \nabla\times\mathbf{H}^{j,\ell}_0(\bz) \cdot\mathbb{E}
\left[\nabla\times\mathbf{U}^{{\rm noise},j,\ell}(\bz)
\overline{\nabla\times\mathbf{U}^{{\rm noise},j,\ell}(\bz')}^T\right]
\mathbf{M}_S^\epsilon\nabla\times\mathbf{H}^{j,\ell}_0(\bz')
\Big\}.
\end{align*}
Using the arguments as in Lemma \ref{lemUnoise}, it can be easily proved that 
\begin{align*}
\mathbb{E}\left[\nabla\times\mathbf{U}^{{\rm noise},j,\ell}(\bz)
\overline{\nabla\times\mathbf{U}^{{\rm noise},j,\ell}(\bz')}^T\right]
\simeq 
-{\sigma_{\rm noise}^2\K}(4\epsilon_0)^{-1}\Im m\Big\{\mathbf{\Gamma}(\bz,\bz')\Big\}.
\end{align*}
Therefore,
\begin{align*}
&{\rm Cov}\big(\partial_T\mathcal{H}_f(\bz),\partial_T\mathcal{H}_f(\bz')\big)
\\
\simeq &-\frac{a_\epsilon^2\sigma^2_{\rm noise}\K}{8n^2\epsilon_0}\sum_{\ell=1}^2\sum_{j=1}^n
\Re e\Big
\{\mathbf{M}_S^\epsilon\nabla\times\mathbf{H}^{j,\ell}_0(\bz)\cdot
\Im m\Big\{\mathbf{\Gamma}(\bz,\bz')\Big\}\mathbf{M}_S^\epsilon
\overline{\nabla\times\mathbf{H}^{j,\ell}_0(\bz')}
\Big\},
\\
=&-\frac{a_\epsilon^2\sigma^2_{\rm noise}\K}{8n^2\epsilon_0}\sum_{\ell=1}^2\sum_{j=1}^n
\Re e\Big
\{\Im m\Big\{\mathbf{\Gamma}(\bz,\bz')\Big\}\mathbf{M}_S^\epsilon
:\mathbf{M}_S^\epsilon\left[\nabla\times\mathbf{H}^{j,\ell}_0(\bz)\right]
\left[\overline{\nabla\times\mathbf{H}^{j,\ell}_0(\bz')}\right]^T
\Big\},
\\
=&-\frac{a_\epsilon^2\sigma^2_{\rm noise}\K^3}{8n\epsilon_0}
\Re e\Big\{\Im m\big\{\mathbf{\Gamma}(\bz,\bz')\big\}\mathbf{M}_S^\epsilon
:\mathbf{M}_S^\epsilon\Big[\frac{1}{n} \sum_{\ell=1}^2\sum_{j=1}^n\left(\theta_j\times\theta_j^{\perp,\ell}\right)
\left(\theta_j\times\theta_j^{\perp,\ell}\right)^T
e^{i\K\theta^T_j(\bz-\bz')}\Big]
\Big\},
\\
\simeq &-\frac{\pi a_\epsilon^2\sigma^2_{\rm noise}\K^2}{2n\epsilon_0^2}
\Re e\left\{\Im m\big\{\mathbf{\Gamma}(\bz,\bz')\big\}\mathbf{M}_S^\epsilon
:\mathbf{M}_S^\epsilon\Im m\big\{\mathbf{\Gamma}(\bz,\bz')\big\}
\right\},
\end{align*}
where the use of the approximation \eqref{theta2} has been made to get last identity. The analysis above elucidates that the conclusions drawn in Section \ref{sec.5.1} are valid for the case of dielectric inclusions as well and functional $\partial_T\mathcal{H}_f$ is robust with respect to measurement noise.

When $D$ is a spherical inclusion, the covariance of the corrupted image turns out to be 
\begin{align*}
{\rm Cov}\big(\partial_T\mathcal{H}_f(\bz),\partial_T\mathcal{H}_f(\bz')\big)
\simeq {\sigma^2_{\rm noise}\widetilde{a}_\epsilon^2\K^2}(2n)^{-1}
\big\|\Im m\big\{\mathbf{\Gamma}(\bz,\bz')\big\}
\big\|^2,
\end{align*}
where
$$
\widetilde{a}_\epsilon
:= \frac{3\sqrt{\pi}\epsilon_2|B_S|}{\epsilon_0(2\epsilon_0+\epsilon_2)}a_\epsilon
=\frac{3\sqrt{\pi}|\epsilon_0-\epsilon_2||B_S|}{\epsilon_0(2\epsilon_0+\epsilon_2)}.
$$
Therefore, the variance of $\partial_T\mathcal{H}_f$  for a spherical dielectric inclusion is given as 
\begin{eqnarray}
{\rm Var}\big(\partial_T\mathcal{H}_f(\bz_S)\big) \simeq
{\sigma^2_{\rm noise}\widetilde{a}_\epsilon^2\K^2}(2n)^{-1}
\big\|\Im m\big\{\mathbf{\Gamma}(\bz_S,\bz_S)\big\}
\big\|^2. 
\end{eqnarray} 

The signal-to-noise ratio in this case can be given as 
\begin{eqnarray}
{\rm SNR}\simeq 
\frac{12\rho^3|B_D|\K\sqrt{2n\pi}|\epsilon_0-\epsilon_1|}{\epsilon_0(2\epsilon_0+\epsilon_1)\sigma_{\rm noise}}\Big\|\Im m\Big\{\mathbf{\Gamma}(\bz_D,\bz_D)\Big\}
\Big\|,
\end{eqnarray}
by virtue of Corollary \ref{cor}.  As in the previous section, the behavior of $j_0$ and $j_2$ when $r\to 0$ suggests that
\begin{eqnarray}
{\rm SNR}\simeq 
\frac{12\sqrt{n}\rho^3|B_D|\K^2|\epsilon_0-\epsilon_1|}{\sqrt{2\pi}(2\epsilon_0+\epsilon_1)\sigma_{\rm noise}}.
\end{eqnarray}

\section{Statistical stability with respect to medium noise}\label{medNoise} 

In this section, we aim to investigate the statistical stability of the imaging functional $\partial_T\mathcal{H}_f$ with respect to medium noise. For simplicity, we assume that only one of the permittivity and permeability parameters fluctuates around the background value at a time.  The general case of medium noise can be dealt with analogously but is more involved and is not presented for brevity.

\subsection{Fluctuations in permeability} \label{subsection1}

Let the permeability of $\Omega$, denoted by $\mu(\bx)$ throughout in this section, be fluctuating around the background permeability such that 
\begin{eqnarray}
\label{mu(x)}
\mu(\bx)=\mu_0\big(1+\gamma(\bx)\big),
\end{eqnarray}
where $\gamma(\bx)$ represents a random fluctuation such that the typical size of $\gamma$, denoted by $\sigma_\gamma$, is small enough so that the Born approximation is valid. We emphasize that $\gamma$ is a real-valued function. 
\newline
\textbf{Nota Bene.}\quad Throughout this subsection, we term the homogeneous medium with parameters $(\epsilon_0,\mu_0)$ as the reference medium, and the random medium without inclusion as the background medium still denoted by $\Omega$ by abuse of notation. Further, superposed $0$ indicates a field in the reference medium and any field otherwise is related to the random medium with or without inclusion henceforth.

Let $\mathbf{G}^0$ and $\mathbf{G}$ be the reference and background dyadic Green's functions with Neumann type boundary conditions, that is, the solutions to
\begin{eqnarray}
\begin{cases}
\nabla_\bx\times\nabla_\bx\times \mathbf{G}^0(\bx,\by)-\K^2\mathbf{G}^0(\bx,\by)=-\epsilon_0\delta_\by(\bx)\mathbf{I}_3, & \Omega,
\\
\left(\nabla_\bx\times\mathbf{G}^0(\bx,\by)\right)\times \nu(\bx)=0, & \partial\Omega, 
\end{cases}\label{G0}
\end{eqnarray}
and
\begin{eqnarray}
\begin{cases}
\nabla_\bx\times\nabla_\bx\times \mathbf{G}(\bx,\by)-(1+\gamma(x))\K^2\mathbf{G}(\bx,\by)=-\epsilon_0\delta_\by(\bx)\mathbf{I}_3, & \Omega,
\\
\left(\nabla_\bx\times\mathbf{G}(\bx,\by)\right)\times \nu(\bx)=0, & \partial\Omega. 
\end{cases}\label{G}
\end{eqnarray}

The following result holds and can be proved by similar arguments as in \cite[Theorem 2.28]{AKpol}.
\begin{lem}\label{Lem6.1}
For all $\bx\in\partial\Omega$ and $\by\in\Omega$, we have 
\begin{eqnarray}
\left(\frac{1}{2}+\mathcal{P}^{\K,0}_*\right)\mathbf{G}^0(\bx,\by)
=\mathbf{\Gamma}^{0}(\bx,\by).
\end{eqnarray}
\end{lem}

The following Born approximation is valid  
\begin{eqnarray}
\mathbf{G}(\bx,\by)= \mathbf{G}^0(\bx,\by)-\frac{\K^2}{\epsilon_0}\int_\Omega \mathbf{G}^0(\bx,\bz)\gamma(\bz)\mathbf{G}^0(\bz,\by) d\bz +o(\sigma_\gamma).\label{born1}
\end{eqnarray}
Moreover, we also have 
\begin{eqnarray}
\mathbf{H}_0(\bx)=\mathbf{H}_0^0(\bx)-\frac{\K^2}{\epsilon_0}\int_\Omega \mathbf{G}^0(\bx,\bz)\gamma(\bz)\mathbf{H}^0_0(\bz) d\bz+o(\sigma_\gamma).\label{born2} 
\end{eqnarray} 

The back-propagator $\mathbf{U}$ is now constructed as follows, 
\begin{equation}
\ds \mathbf{U}(\bz)= -\frac{1}{\epsilon_0}\int_{\partial\Omega} \mathbf{\Gamma}^{0} (\bx,\bz)\nu(\bx)\times\overline{\left(\frac{1}{2}\mathcal{I}-\mathcal{P}^{\K,0}\right)\left[(\mathbf{H}_\rho-\mathbf{H}_0^0)\times\nu\right](\bx)} d\sigma(\bx).
\end{equation}
Note that the back-propagation step uses reference fundamental solution and the reference magnetic solution since the background solutions are unknown. This substantiates that the back-propagation step transports not only the true scattered field but also the first scattering source (under Born approximation) due to fluctuations, thereby generating a spatially distributed contribution in the image. Further, the background Green's function $\mathbf{G}$ is not known exactly but up to a first order approximation. Therefore, the back-propagation using reference Green's function $\mathbf{G}^0$ may affect the principle peak of the imaging functional around $\bz_S\simeq\bz_D$.

We express $\mathbf{H}_\rho-\mathbf{H}_0^0$ as the sum of two terms $\mathbf{H}_\rho-\mathbf{H}_0$ and $\mathbf{H}_0-\mathbf{H}_0^0$   and subsequently invoke Lemma \ref{Asymp} and Born approximations \eqref{born1}--\eqref{born2}. Therefore,
\begin{align*}
\mathbf{U}(\bz)&= -\frac{1}{\epsilon_0}\int_{\partial\Omega}\mathbf{\Gamma}^{0}(\bx,\bz)\nu(\bx)\times\overline{\left(\frac{1}{2}\mathcal{I}-\mathcal{P}^{\K,0}\right)\left[(\mathbf{H}_\rho-\mathbf{H}_0)\times\nu\right](\bx)} d\sigma(\bx)
\\
&-\frac{1}{\epsilon_0}\int_{\partial\Omega}\mathbf{\Gamma}^{0}(\bx,\bz)\nu(\bx)\times\overline{\left(\frac{1}{2}\mathcal{I}-\mathcal{P}^{\K,0}\right)\left[(\mathbf{H}_0-\mathbf{H}_0^0)\times\nu\right](\bx)} d\sigma(\bx),
\\
&=
-\frac{1}{\epsilon_0}\int_{\partial\Omega}\mathbf{\Gamma}^{0}(\bx,\bz)\nu(\bx)\times\overline{\left(\frac{1}{2}\mathcal{I}-\mathcal{P}^{\K,0}\right)\left[(\mathbf{H}_\rho^0-\mathbf{H}_0^0)\times\nu\right](\bx)} d\sigma(\bx)
\\
&
+\frac{\K^2}{\epsilon_0^2} \int_{\partial\Omega}\mathbf{\Gamma}^{0}(\bx,\bz)\nu(\bx)\times\overline{\left(\frac{1}{2}\mathcal{I}-\mathcal{P}^{\K,0}\right)\left[\int_\Omega \mathbf{G}^0(\cdot,\by)\gamma(\by)(\mathbf{H}_\rho^0-\mathbf{H}_0^0)(\by)d\by\times\nu\right](\bx)} d\sigma(\bx)
\\
&
+\frac{\K^2}{\epsilon_0^2}\int_{\partial\Omega}\mathbf{\Gamma}^{0}(\bx,\bz)
\nu(\bx)\times\overline{\left(\frac{1}{2}\mathcal{I}-\mathcal{P}^{\K,0}\right)\left[\int_\Omega\mathbf{G}^0(\cdot,\by)\gamma(\by)\mathbf{H}^0_0(\by)d\by\times \nu\right](\bx)}d\sigma(\bx)
+o(\sigma_\gamma),
\\
&
= T_1+T_2+T_3+o(\sigma_\gamma),
\end{align*}
where $T_1$, $T_2$ and $T_3$ represent the first, second and third term on the right hand side respectively. 

Note that $T_1$ is exactly the reference back-propagator defined in \eqref{U}. Therefore, we will denote this term by $\mathbf{U}^{\rm true}(\bz)$. From \cite[Theorem 2.1]{AV}, we have $(\mathbf{H}_\rho^0-\mathbf{H}_0^0)=O(\rho^3)$. Consequently, the second term $T_2$
is of the order $O(\sigma_\gamma \rho^3)$ and is neglected henceforth. Finally, by using Lemma \ref{Lem2.4}, Lemma \ref{Lem6.1} and Lemma \ref{HKI-lem1} respectively,  we have
\begin{align*}
T_3=& \frac{\K^2}{\epsilon^2_0}\int_{\partial\Omega} \left(\mathbf{\Gamma}^0(\bx,\bz)\times\nu(\bx)\right)^T 
\overline{\left(\frac{1}{2}\mathcal{I}+\mathcal{P}_*^{\K,0}\right)
\left[\int_\Omega\mathbf{G}^0(\cdot,\by)\gamma(\by)\mathbf{H}^0_0(\by) d\by\right](\bx)}\times\nu(\bx) d\sigma(\bx),
\\
=&-\frac{\K^2}{\epsilon^2_0}
\int_{\partial\Omega}\left(\mathbf{\Gamma}^0(\bx,\bz)\times\nu(\bx)\times\nu(\bx)\right)^T \int_\Omega \overline{\mathbf{\Gamma}^0(\bx,\by)}\gamma(\by)\overline{\mathbf{H}^0_0(\by)}d\by d\sigma(\bx),
\\
=&\frac{\K^2}{\epsilon^2_0}\int_\Omega \gamma(\by)\left[\int_{\partial\Omega}\mathbf{\Gamma}^0(\bx,\bz) \overline{\mathbf{\Gamma}^0(\bx,\by)}^T d\sigma(\bx)\right]\overline{\mathbf{H}^0_0(\by)}d\by,
\\
\simeq 
&-\frac{\K}{\epsilon_0}\int_\Omega \gamma(\by)\Im m\big\{\mathbf{\Gamma}^0(\by,\bz) \big\}\overline{\mathbf{H}^0_0(\by)}d\by.
\end{align*}
Therefore, we conclude that 
$
\mathbf{U}(\bz)=\mathbf{U}^{\rm true}(\bz) +\mathbf{U}^{\rm noise}(\bz) +O(\sigma_\gamma\rho^3)+o(\sigma_\gamma),
$
where $\mathbf{U}^{\rm noise}$ is defined by 
\begin{eqnarray}
\mathbf{U}^{\rm noise}(\bz):=-\frac{\K}{\epsilon_0}\int_\Omega \gamma(\by)\Im m\big\{\mathbf{\Gamma}^0(\by,\bz) \big\}\overline{\mathbf{H}^0_0(\by)}d\by.\label{UnoiseApp}
\end{eqnarray}

The expansion of $\mathbf{U}(\bz)$ clearly shows that the back-propagator in the random medium is approximately the sum of  reference back-propagator and the error term due to clutter. The reference back propagator $\mathbf{U}^{\rm true}$ produces the principle peak of $\partial_T\mathcal{H}_f$, that is, without medium noise. The back-propagator $\mathbf{U}^{\rm noise}$ generates a speckle field corrupting the reconstructed image. In rest of this subsection,   we restrict ourselves to the case of permeable inclusions and dielectric inclusions for simplicity in order to analyze the speckle field generated by $\mathbf{U}^{\rm noise}$. Further, the situation when there are  multiple incident fields of the form \eqref{Hj} is taken into account.  

\subsubsection{Speckle field analysis for permeable inclusions}\label{Sec.6.1.1}

Let us  compute the covariance of  speckle field due to back-propagation of $\mathbf{U}^{\rm noise, j, \ell}$. We have
\begin{align}
{\rm Cov}&\Big(\partial_T\mathcal{H}_f(\bz), \partial_T\mathcal{H}_f(\bz')\Big)=
\frac{\K^4a_\mu^2}{n^2}\sum_{\ell,\ell'=1}^2\sum_{j,j'=1}^n
\nonumber
\\
&
\mathbb{E}\left[
\Re e\left\{\mathbf{U}^{{\rm noise},j,\ell}(\bz)\cdot\mathbf{M}_{S}^\mu\mathbf{H}^{0,j,\ell}_0(\bz) \right\}
\Re e\left\{\mathbf{U}^{{\rm noise},j',\ell'}(\bz')\cdot \mathbf{M}_{S}^\mu\mathbf{H}^{0,j',\ell'}_0(\bz')\right\}
\right],
\end{align}
for all $\bz,\bz'\in\Omega$, where $\mathbf{H}^{0,j,\ell}_0$ are the incident fields of the form \eqref{Hj}. 
First of all, we invoke \eqref{UnoiseApp}, \eqref{Hj} and \eqref{theta1} to get 
\begin{align*}
\frac{1}{n}\sum_{\ell=1}^2\sum_{j=1}^n 
\mathbf{U}^{{\rm noise},j,\ell}
&(\bz) \cdot\mathbf{M}_{S}^\mu\mathbf{H}^{0,j, \ell}_0(\bz)
\\
= &
-\frac{\K}{\epsilon_0 n}\sum_{\ell=1}^2\sum_{j=1}^n\int_\Omega \gamma(\by)\Im m\left\{\mathbf{\Gamma}^0(\by,\bz)\right\}\overline{\mathbf{H}^{0,j,\ell}_0(\by)}d\by\cdot
\mathbf{M}_{S}^\mu\mathbf{H}^{0,j,\ell}_0(\bz),
\\
= &
-\frac{\K}{\epsilon_0 n}\sum_{\ell=1}^2\sum_{j=1}^n\int_\Omega \gamma(\by)\Im m\left\{\mathbf{\Gamma}^0(\by,\bz)\right\}:\mathbf{M}_{S}^\mu\mathbf{H}^{0,j,\ell}_0 (\bz) \left[\overline{\mathbf{H}^{0,j,\ell}_0(\by)}\right]^T d\by,
\\
= &
-\frac{\K}{\epsilon_0}\int_\Omega \gamma(\by)\Im m\left\{\mathbf{\Gamma}^0(\by,\bz)\right\}:\mathbf{M}_{S}^\mu
\Bigg[
\frac{1}{n}\sum_{\ell=1}^2\sum_{j=1}^n\theta_j^{\perp,\ell}\Big[\theta_j^{\perp,\ell}\Big]^T
e^{i\K\theta_j^T(\bz-\by)}
\Bigg] d\by,
\\
\simeq &
\phantom{-}\frac{4\pi}{\epsilon_0^2}\int_\Omega \gamma(\by)\mathcal{Q}_\gamma[\mathbf{M}_{S}^\mu](\by,\bz)d\by,
\end{align*}
where $\mathcal{Q}_\gamma$ is a non-negative real valued function defined for any $3\times 3$ real matrix $\mathbf{A}$ by
$$
\mathcal{Q}_\gamma[\mathbf{A}](\by,\bz):=\Im m\big\{\mathbf{\Gamma}^0(\by,\bz)\big\}:\mathbf{A} \Im m\big\{\mathbf{\Gamma}^0(\by,\bz)\big\}.
$$
Then, the covariance of the speckle field can be approximated by 
$$
{\rm Cov}\Big(\partial_T\mathcal{H}_f(\bz), \partial_T\mathcal{H}_f(\bz')\Big)
\simeq {16\pi^2 a_\mu^2 \K^4}\epsilon_0^{-4} \iint_{\Omega\times\Omega} C_\mu(\by,\by')\mathcal{Q}_\gamma[\mathbf{M}_{S}^\mu](\by,\bz)\mathcal{Q}_\gamma[\mathbf{M}_{S}^\mu](\by',\bz') d\by d\by',
$$
where $C_\gamma(\by,\by')=\mathbb{E}\left[\gamma(\by)\gamma(\by')\right]$  is  the two-point correlation function of the fluctuations in permeability.  The function $\bz_S\to \mathcal{Q}_\gamma[\mathbf{M}_S^\mu](\bz_S,\bz_D)$ is maximal for $\bz_S=\bz_D$ and the focal spot of its peak is of the order of half the operating wavelength.

Note that $\mathcal{Q}_\gamma[\mathbf{I}_3](\by,\bz)= \|\Im m\big\{\mathbf{\Gamma}(\by,\bz)\big\}\|^2$. Therefore, for the case of a spherical inclusion, thanks to Lemma \ref{PT}, we have
\begin{align}
\frac{1}{n}\sum_{\ell=1}^2\sum_{j=1}^n \mathbf{U}^{{\rm noise},j,\ell}(\bz)
\cdot \mathbf{M}_{S}^\mu\mathbf{H}^{0,j,\ell}_0(\bz) 
\simeq
\frac{b_\mu}{a_\mu}\int_\Omega \gamma(\by)\mathcal{Q}_\gamma[\mathbf{I}_3](\by,bz)d\by.\label{Approx2}
\end{align}
\begin{align}
{\rm Cov}\Big(\partial_T\mathcal{H}_f(\bz), \partial_T\mathcal{H}_f(\bz')\Big)
\simeq b_\mu^2 \K^4\iint_{\Omega\times\Omega} C_\mu(\by,\by')
\mathcal{Q}_\gamma[\mathbf{I}_3](\by,\bz)
\mathcal{Q}_\gamma[\mathbf{I}_3](\by',\bz') d\by d\by',
\label{ApproxCov}
\end{align}
where 
\begin{eqnarray}
b_\mu:= \frac{12\pi\left(\mu_0-\mu_2\right)|B_S|}{\epsilon_0^2(2\mu_0+\mu_2)}.
\end{eqnarray}

The expression \eqref{Approx2} elucidates that the speckle field in the image is essentially the medium noise smoothed by an integral kernel of the form $\|\Im m\{\mathbf{\Gamma}^0\}\|^2$ . Similarly, \eqref{ApproxCov} elucidates that the correlation structure of the speckle field is essentially that of the medium
noise smoothed by the same kernel. Since the typical width of $\Im m\{\mathbf{\Gamma}^0\}$ is about half the
wavelength, the correlation length of the speckle field is roughly the maximum between the correlation length of medium noise and the wavelength, that is, of the same order as the main peak centered at location $\bz_S\simeq \bz_D$. Thus, there is no way to distinguish the main peak from the hot spots of the speckle field based on their shapes. Only the height of the main peak can allow it to be visible out of the speckle field.  Unlike measurement noise case discussed in the previous section, the factor $\sqrt{n}$ disappeared. Therefore, the functional $\partial_T\mathcal{H}_f$ is moderately stable with respect to medium noise. Moreover, the main peak of $\partial_T\mathcal{H}_f$ is affected by the clutters, unlike in the measurement noise case. Thus, $\partial_T\mathcal{H}_f$ is more robust with respect to measurement noise than medium noise.

\subsubsection{Speckle field analysis for dielectric inclusions}

In order to compute the covariance of the speckle field generated by the back-propagation of $\mathbf{U}^{\rm noise}$ for a dielectric inclusion, we first note that
\begin{align*}
\frac{1}{n}\sum_{\ell=1}^2\sum_{j=1}^n 
& \nabla\times\mathbf{U}^{{\rm noise},j,\ell}(\bz) \cdot\mathbf{M}_S^\epsilon\nabla\times\mathbf{H}^{0,j,\ell}_0(\bz)
\\
\simeq & -
\frac{\K}{\epsilon_0 n}\sum_{\ell=1}^2\sum_{j=1}^n\int_{\Omega}\gamma(\by) \nabla_\bz\times\Im m\Big\{\mathbf{\Gamma}^0(\by,\bz)\Big\}\overline{\mathbf{H}_0^{0,j,\ell}(\by)}\cdot \mathbf{M}_S^\epsilon\nabla\times\mathbf{H}_0^{0,j,\ell}(\bz)d\by, 
\\
= &-
\frac{\K}{\epsilon_0 n} \sum_{\ell=1}^2\sum_{j=1}^n\int_{\Omega}\gamma(\by)\nabla_\bz\times\Im m\Big\{\mathbf{\Gamma}^0(\by,\bz)\Big\}:\mathbf{M}_S^\epsilon \nabla_\bz\times\left(\mathbf{H}_0^{0,j,\ell}(\bz)\overline{\mathbf{H}_0^{0,j,\ell}(\by)}^T\right) d\by,
\\
= &
- \frac{\K}{\epsilon_0} \int_{\Omega}\gamma(\by)\nabla_\bz\times\Im m\Big\{\mathbf{\Gamma}^0(\by,\bz)\Big\}:\mathbf{M}_S^\epsilon \nabla_\bz\times \Bigg(\frac{1}{n}\sum_{\ell=1}^2\sum_{j=1}^n\theta_j^{\perp,\ell}\Big(\theta_j^{\perp,\ell}\Big)^Te^{i\K\theta_j^T(\bz-\by)}\Bigg) d\by,
\\
\simeq &\phantom{-}\frac{4\pi}{\epsilon_0^2} \int_{\Omega}\gamma(\by)\widetilde{\mathcal{Q}}_\gamma
[\mathbf{M}^\epsilon_S](\by,\bz)d\by,
\end{align*} 
where $\widetilde{\mathcal{Q}}_\gamma$ is a non-negative real valued function defined for any $3\times 3$ real matrix $\mathbf{A}$ by
$$
\widetilde{\mathcal{Q}}_\gamma[\mathbf{A}](\by,\bz)
=
\nabla_\bz\times\Im m\Big\{\mathbf{\Gamma}^0(\by,\bz)\Big\}:\mathbf{M}_S^\epsilon \nabla_\bz\times\Im m\Big\{\mathbf{\Gamma}^0(\by,\bz)\Big\}d\by.
$$
Therefore, the covariance turns out to be
$$
{\rm Cov}\Big(\partial_T\mathcal{H}_f(\bz), \partial_T\mathcal{H}_f(\bz')\Big)
\simeq\frac{16\pi^2 a_\epsilon^2}{\epsilon_0^4}\iint_{\Omega\times\Omega} C_\mu(\by,\by')\widetilde{\mathcal{Q}}_\gamma[\mathbf{M}_{S}^\epsilon](\by,\bz)\widetilde{\mathcal{Q}}_\gamma[\mathbf{M}_{S}^\epsilon](\by',\bz') d\by d\by',
$$
Finally, note that $\widetilde{\mathcal{Q}}_\gamma[\mathbf{I}_3](\by,\bz)=\|\Im m\left\{\mathbf{\Gamma}^0(\by,\bz)\right\}\|^2$, thus for a spherical dielectric inclusion
\begin{align}
\frac{1}{n}\sum_{\ell=1}^2\sum_{j=1}^n \nabla\times \mathbf{U}^{{\rm noise},j,\ell}(\bz)\cdot\mathbf{M}_S^\epsilon\nabla\times\mathbf{H}_0^{0,j,\ell}(\bz)
\simeq 
\frac{b_\epsilon}{a_\epsilon}\int_\Omega \gamma(\by)\widetilde{\mathbf{Q}}_\gamma[\mathbf{I}_3](\by,\bz)d\by,
\end{align}
\begin{align}
{\rm Cov}\Big(\partial_T\mathcal{H}_f(\bz), \partial_T\mathcal{H}_f(\bz')\Big)
\simeq {b}_\epsilon^2
\iint_{\Omega\times\Omega} {C}_\gamma(\by,\by')  
\widetilde{\mathcal{Q}}_\gamma[\mathbf{I}_3](\by,\bz)
\widetilde{\mathcal{Q}}_\gamma[\mathbf{I}_3](\by',\bz')
d\by d\by ',\label{Cov4}
\end{align}
where the constant ${b}_\epsilon$ is defined by
\begin{eqnarray}
{b}_\epsilon=\frac{12\pi(\epsilon_0-\epsilon_2)|B_S|}{\epsilon_0^2(2\epsilon_0+\epsilon_2)}.
\end{eqnarray}
The conclusions drown in Section \ref{Sec.6.1.1} still hold in this case and the imaging functional is moderately stable. 

\subsection{Fluctuations in permittivity} 

Let us now investigate the stability of the imaging framework with respect to medium noise when the permittivity, hereafter denoted by $\epsilon$, is fluctuating randomly around the reference permittivity. We assume that the fluctuating  \emph{background} permittivity is such that 
\begin{eqnarray}
{\epsilon^{-1}(\bx)}:= {\epsilon_0^{-1}}[1+\alpha(\bx)],
\end{eqnarray} 
where $\alpha$ is a random fluctuation. It is again assumed that the fluctuation is weak so that the Born approximation is appropriate. We will make  use of the same conventions as in Section \ref{subsection1} for reference and background media, and fields. 

The equation for the magnetic field with fluctuating permittivity is then given by 
\begin{eqnarray}
\nabla\times \nabla\times \mathbf{H}_\rho(\bx) -\K^2\mathbf{H}_\rho(\bx)= -\nabla\times\alpha(\bx)\nabla\times\mathbf{H}_\rho(\bx).
\end{eqnarray}
Since  the Born approximation is appropriate thanks to assumption of weak fluctuations, we have $\mathbf{H}_\rho \simeq \mathbf{H}_\rho^0-\mathbf{H}_\rho^1$, where $\mathbf{H}_\rho^0$ solves the reference problem and $\mathbf{H}_\rho^1$ solves 
\begin{eqnarray}
\nabla\times \nabla\times \mathbf{H}^1_\rho(\bx)-\K^2\mathbf{H}^1_\rho(\bx)
= -\nabla\times\alpha(\bx)\nabla\times\mathbf{H}^0_\rho(\bx).
\end{eqnarray}
Consequently, we have 
\begin{eqnarray}
\mathbf{H}^1_\rho(\bx)= \frac{1}{\epsilon_0}\int_\Omega\mathbf{G}^0(\bx,\by)
\left(\nabla \times\alpha(\by)\nabla\times
\mathbf{H}^0_\rho(\by)\right)d\by,
\end{eqnarray}
where $\mathbf{G}^0(\bx,\by)$ is given by \eqref{G0}.

Following the analysis in Section \ref{subsection1}, it can be noticed that the back-propagator, again defined in terms of the reference fundamental solution and associated reference solution, consists of two terms, one leading to the true image whereas the second giving rise to a speckle field corrupting the image thanks to permittivity fluctuations. Using analogous arguments and manipulations as in the permeability fluctuation case, the noise back-propagating term turns out to be
\begin{align}
\mathbf{U}^{\rm noise}(\bz)
\simeq 
& -\frac{1}{\epsilon_0\K}\int_\Omega\Im m\Big\{\mathbf{\Gamma}^0(\by,\bz)\Big\}\nabla\times\alpha(\by)\nabla
\times\overline{\mathbf{H}^0_0(\by)}d\by.
\end{align}

\subsubsection{Speckle field analysis for permeable inclusions}

For a permeable inclusion, the speckle field generated by $\partial_T\mathcal{H}_f$  at $\bz\in\Omega$ is given by 
\begin{align}
\nonumber
T_4:=&\frac{1}{n}\sum_{\ell=1}^2
\sum_{j=1}^n 
\Re e\Big\{
\mathbf{U}^{{\rm noise},j,\ell}(\bz)\cdot \mathbf{M}_S^\mu\mathbf{H}_0^{0,j,\ell}(\bz)
\Big\}
\\
\simeq &
-\frac{1}{\K \epsilon_0 n}\sum_{j=1}^n 
\Re e\left\{
\int_\Omega \Im m\Big\{\mathbf{\Gamma}^0(\by,\bz)\Big\}\nabla\times\alpha(\by)\nabla
\times\overline{\mathbf{H}^{0,j,\ell}_0(\by)}\cdot  
\mathbf{M}_S^\mu \mathbf{H}_0^{0,j,\ell}(\bz)d\by
\right\}
\end{align}
Since $\Im\big\{\mathbf{\Gamma}^0(\by,\bz)\big\}$ is symmetric, we have 
\begin{align}
\nonumber
T_4
\simeq &
-\frac{1}{\K\epsilon_0 n}\sum_{\ell=1}^2\sum_{j=1}^n 
\Re e\left\{
\int_\Omega \Im m\Big\{\mathbf{\Gamma}^0(\by,\bz)\Big\} \mathbf{M}_S^\mu\mathbf{H}_0^{0,j,\ell}(\bz)\cdot\nabla\times\alpha(\by)\nabla
\times\overline{\mathbf{H}^{0,j,\ell}_0(\by)}d\by
\right\}.
\end{align}
Further, on assuming that $\alpha(\bx)=0$ for all $\bx$ in the neighborhood of boundary $\partial\Omega$ and using the Green's theorem, the above expression simplifies to 
\begin{align}
\nonumber
T_4 \simeq & -\frac{1}{\K\epsilon_0 n}\sum_{\ell=1}^2\sum_{j=1}^n 
\Re e\left\{
\int_\Omega \nabla_\by\times\Im m\Big\{\mathbf{\Gamma}^0(\by,\bz)\Big\} \mathbf{M}_S^\mu \mathbf{H}_0^{0,j,\ell}(\bz)\cdot\alpha(\by)\nabla
\times\overline{\mathbf{H}^{0,j,\ell}_0(\by)}d\by
\right\}.
\end{align}
After straight forward calculations and the use of approximation \eqref{theta1}
\begin{align}
\nonumber
T_4\simeq \frac{4\pi}{\K^2\epsilon_0^2}
\Re e\left\{ \int_\Omega \alpha(\by)\mathcal{Q}_\alpha[\mathbf{M}_S^\mu](\by,\bz)d\by 
\right\},
\end{align}
where for any $3\times 3$ real matrix $\mathbf{A}$, the  real valued function $\mathcal{Q}_\alpha$ is defined by  
$$
\mathcal{Q}_\alpha[\mathbf{A}](\by,\bz):=\nabla_\by\times\Im m\big\{\mathbf{\Gamma}^0(\by,\bz)\big\}\mathbf{A}:\nabla_\by\times\Im m\big\{\mathbf{\Gamma}^0(\by,\bz)\big\}.
$$
Consequently, the covariance of the speckle field turns out to be
\begin{eqnarray}
{\rm Cov}\Big(\partial_T\mathcal{H}_f(\bz), \partial_T\mathcal{H}_f(\bz')\Big)\simeq \frac{16\pi^2a_\mu^2}{\epsilon_0^4}\iint_{\Omega\times\Omega} C_\alpha(\by,\by')\mathcal{Q}_\alpha[\mathbf{M}_S^\mu](\by,\bz)\mathcal{Q}_\alpha[\mathbf{M}_S^\mu](\by',\bz')d\by d\by'. 
\end{eqnarray}
Moreover, since $\mathcal{Q}_\alpha[\mathbf{I}_3](\by,\bz)=\big\|\nabla_\by\times\Im m\big\{\mathbf{\Gamma}^0(\by,\bz)\big\}\big\|^2$, for a spherical inclusion
\begin{eqnarray}
{\rm Cov}\Big(\partial_T \mathcal{H}_f(\bz), \partial_T \mathcal{H}_f(\bz')\Big)
=b_\mu\iint_{\Omega\times\Omega} {C}_\alpha(\by,\by')\mathcal{Q}_\alpha[\mathbf{I}_3](\by,\bz)\mathcal{Q}_\alpha[\mathbf{I}_3](\by',\bz')d\by d\by',\label{Cov3}
\end{eqnarray}
where $C_\alpha(\by,\by'):=\mathbb{E}[\alpha(\by)\alpha(\by')]$ is the two point correlation of fluctuation $\alpha$. The expression \eqref{Cov3} is very similar to that studied in \eqref{Cov4}. As already pointed out in Section \ref{Sec.6.1.1}, the speckle field is indeed the medium noise smoothed with an integral kernel whose width is of the order of wavelength.

\subsubsection{Speckle field analysis for dielectric inclusions}

In this case, the speckle field generated by   $\partial_T\mathcal{H}_f$  at $\bz\in\Omega$ is given by 
\begin{align*}
T_5 :=& \frac{1}{n}\sum_{\ell=1}^2\sum_{j=1}^n 
\Re e\Big\{
\nabla\times\mathbf{U}^{{\rm noise},j,\ell}(\bz)\cdot  \mathbf{M}_S^\epsilon\nabla\times\mathbf{H}_0^{0,j,\ell}(\bz)
\Big\},
\\
= &
-\frac{1}{\K\epsilon_0 n}\sum_{\ell=1}^2\sum_{j=1}^n 
\Re e\left\{
\int_\Omega \nabla_\bz\times\Im m\Big\{\mathbf{\Gamma}^0(\by,\bz)\Big\}\nabla\times\alpha(\by)\nabla
\times\overline{\mathbf{H}^{0,j,\ell}_0(\by)}\cdot  \mathbf{M}_S^\epsilon\nabla\times \mathbf{H}_0^{0,j,\ell}(\bz)d\by
\right\},
\\
= &
-\frac{1}{\K \epsilon_0 n}\sum_{\ell=1}^2\sum_{j=1}^n 
\Re e\left\{
\int_\Omega\nabla_\by\times\Im m\Big\{\mathbf{\Gamma}^0(\by,\bz)\Big\}\left(\mathbf{M}_S^\epsilon\nabla\times \mathbf{H}_0^{0,j,\ell}(\bz)\right)\cdot \nabla\times\alpha(\by)\nabla
\times\overline{\mathbf{H}^{0,j,\ell}_0(\by)}
d\by
\right\}.
\end{align*}
Letting $\alpha$ to be zero near $\partial\Omega$ and using Green's theorem, we simplify the above expression to
\begin{align*}
T_5 \simeq &
-\frac{1}{\K \epsilon_0 n}\sum_{\ell=1}^2\sum_{j=1}^n 
\Re e\left\{
\int_\Omega \nabla_\by\times \nabla_\by\times\Im m\Big\{\mathbf{\Gamma}^0(\by,\bz)\Big\}\mathbf{M}_S^\epsilon \nabla\times \mathbf{H}_0^{0,j,\ell}(\bz)\cdot \alpha(\by)\nabla
\times\overline{\mathbf{H}^{0,j,\ell}_0(\by)} d\by \right\},
\\
= &
-\frac{\K}{\epsilon_0 n}\sum_{\ell=1}^2\sum_{j=1}^n 
\Re e\left\{
\int_\Omega \Im m\Big\{\mathbf{\Gamma}^0(\by,\bz)\Big\}\mathbf{M}_S^\epsilon\cdot \alpha(\by)\nabla \times\overline{\mathbf{H}^{0,j,\ell}_0(\by)}\left(\nabla\times \mathbf{H}_0^{0,j,\ell}(\bz)\right)^T d\by
\right\}.
\end{align*} 
Finally, invoking approximation \eqref{theta2}, we arrive at
\begin{align}
T_5 \simeq & \frac{4\pi\K^2}{\epsilon_0^2} 
\int_{\Omega}\gamma(\by)\widetilde{\mathcal{Q}}_\alpha
[\mathbf{M}^\epsilon_S](\by,\bz)d\by,
\end{align} 
where $\widetilde{\mathcal{Q}}_\alpha$ is a non-negative real valued function defined for any $3\times 3$ real matrix $\mathbf{A}$ by
$$
\widetilde{\mathcal{Q}}_\alpha[\mathbf{A}](\by,\bz)
=
\Im m\Big\{\mathbf{\Gamma}^0(\by,\bz)\Big\}\mathbf{M}_S^\epsilon: \Im m\Big\{\mathbf{\Gamma}^0(\by,\bz)\Big\}d\by.
$$
and as a consequence, 
\begin{align}
{\rm Cov}\Big(\partial_T&\mathcal{H}_f(\bz), \partial_T \mathcal{H}_f(\bz')\Big)
\simeq 
b_\epsilon^2\K^4\iint_{\Omega\times\Omega} C_\alpha(\by,\by')\widetilde{\mathcal{Q}}_\alpha[\mathbf{M}_S^\epsilon](\by,\bz)\widetilde{\mathcal{Q}}_\alpha[\mathbf{M}_S^\epsilon](\by',\bz').\label{Cov5}
\end{align}
The results for the a spherical dielectric inclusion are evident from the previous analysis.

\section{Conclusions}\label{conc}

In this paper, we investigated a topological derivative based electromagnetic inclusion detection algorithm using the measurements of the tangential components of scattered magnetic field, considering a full Maxwell equations setting. It is elucidated that the topological derivative based imaging functional behaves like the square of the imaginary part of a free space fundamental magnetic solution and attains its maximum at the true location of the inclusion with Rayleigh resolution limit. The detection algorithm is proved to be very stable with respect to measurement noise and moderately stable with respect to medium noise. Moreover, it is indicated that multiple incident waves significantly enhance the stability of the functional. Albeit, the case of a single inclusion is discussed herein, the results extend to the case of multiple inclusions with a common characteristic size.

\bigskip

\appendix

\section{Proof of Lemma \ref{HKI-lem2}}\label{Append.HK2}

We recall from \cite[Lemma 3.1]{CCH}, that for all constant vectors $\mathbf{p}, \mathbf{q}\in\RR^3$ and $\bx,\by\in\mathbf{B}(0,r)$
\begin{align}
&2i\epsilon_0\mathbf{p}\cdot\Im m\Big\{\mathbf{\Gamma}(\bx,\bz)\Big\}\mathbf{q}
\nonumber
\\
=&\int_{\partial\mathbf{B}}
\Big(
\overline{\mathbf{\Gamma}(\bx,\bz)}\mathbf{p}\cdot\nu(\bz)\times\nabla_z\times\mathbf{\Gamma}(\bz,\by)\mathbf{q}
-\nu(\bz)\times\nabla_z\times\overline{\mathbf{\Gamma}(\bx,\bz)}\mathbf{p}\cdot\mathbf{\Gamma}(\bz,\by)\mathbf{q}\Big)d\sigma(\bz),
\nonumber
\\
=&\int_{\partial\mathbf{B}}
\Big(
\left[\overline{\mathbf{\Gamma}(\bx,\bz)}\mathbf{p}\times\nu(\bz)\right]\cdot\left[\nabla_z\times\mathbf{\Gamma}(\bz,\by)\mathbf{q}\right]
-\left[\nabla_z\times\overline{\mathbf{\Gamma}(\bx,\bz)}\mathbf{p}\right]\cdot\left[\mathbf{\Gamma}(\bz,\by)\mathbf{q}\times \nu(\bz)\right]\Big)d\sigma(\bz).\label{e0}
\end{align} 
Moreover, in the far field where $r\to \infty $, we have
\begin{eqnarray}
\mathbf{\Gamma}(\bx,\by)\mathbf{p}&=&O(r^{-1}), \label{e1}
\\
\frac{\partial}{\partial x_j }\mathbf{\Gamma}(\bx,\by)\mathbf{p}&=&O(r^{-1}), \label{e2}
\\
\nabla_\bx\times\mathbf{\Gamma}(\bx,\by)\mathbf{p}+i\K\mathbf{\Gamma}(\bx,\by)\mathbf{p}\times\nu(\bx)&=&O(r^{-2}), \label{e3}
\\
\frac{\partial}{\partial x_j}\Big(\nabla_\bx\times\mathbf{\Gamma}(\bx,\by)\mathbf{p}+i\K\mathbf{\Gamma}(\bx,\by)\mathbf{p}\times\nu(\bx)\Big)&=&O(r^{-2}).\label{e4}
\end{eqnarray}

By virtue of the estimates \eqref{e1} and \eqref{e3}, the expression \eqref{e0} renders
\begin{equation}
\int_{\partial\mathbf{B}(0,r)}\Big(\overline{\mathbf{\Gamma}(\bx,\bz)}\times\nu(\bz)\Big)^T\Big(\mathbf{\Gamma}(\bz,\by)\times\nu(\bz)\Big)d\sigma(\bz) = -\frac{\epsilon_0}{\K}\Im m\big\{\mathbf{\Gamma}(\bx,\by)\big\}+O(r^{-1}).
\end{equation} 
The above relation also shows that  $|\widetilde{q}_{ij}(\bx,\by)|= O(r^{-1})$. The estimate for  $|\nabla_\bx\widetilde{q}_{ij}(\bx,\by)|$ can be proved analogously using \eqref{e3} and \eqref{e4}. This completes the proof. 

\section{Proof of Lemma \ref{lemUnoise}}\label{App.Proof5.1}
 
First of all note that, since $\etab$ is a mean-zero circular Gaussian random process, $\mathbf{U}^{\rm noise}(\bz)$ is also a mean-zero circular Gaussian random process thanks to linearity. Moreover, its covariance can be calculated for all $\bz,\bz'\in\Omega$ as 
\begin{align}
\mathbb{E}\Big[\mathbf{U}^{\rm noise}(\bz)
\overline{\mathbf{U}^{\rm noise}(\bz')}^T\Big]
:=&\frac{1}{\epsilon_0^2}\mathbb{E}\Big[\int_{\partial\Omega} \mathbf{\Gamma}(\bx,\bz)\nu(\bx)\times
\overline{\left(\frac{1}{2}\mathcal{I}-\mathcal{P}^\K\right)
\left[\etab\times\nu\right](\bx)} d\sigma(\bx)
\nonumber
\\
&\qquad
\left(
\int_{\partial\Omega}\overline{\mathbf{\Gamma}(\bx',\bz')}\nu(\bx')\times
\left(\frac{1}{2}\mathcal{I}-\mathcal{P}^\K\right)
\left[\etab\times\nu\right](\bx') d\sigma(\bx')\right)^T
\Big],
\nonumber
\\
=& \frac{1}{\epsilon_0^2}\sum_{p,q=1}^2\mathbb{E}_{pq}(\bz,\bz'),\label{Cov1}
\end{align}
where $\mathbb{E}_{pq}:= \mathbb{E}\left[\mathcal{J}_\alpha(\bz)
\overline{\mathcal{J}_\beta(\bz')}\right]$ for all $p,q\in\{1,2\}$ with 
\begin{align}
\mathcal{J}_1(\bz):=&
\frac{1}{2}\int_{\partial\Omega} \left(\mathbf{\Gamma}(\bx,\bz)\times\nu(\bx)\right)^T
\left(\overline{\etab(\bx)} \times\nu(\bx)\right)d\sigma(\bx),
\\
\mathcal{J}_2(\bz):=&\int_{\partial\Omega} \left(\mathbf{\Gamma}(\bx,\bz)\times\nu(\bx)\right)^T
\overline{\mathcal{P}^\K\left[\etab\times\nu\right](\bx)}d\sigma(\bx).
\end{align}

Let us now analyze each term individually. Note that 
\begin{align}
\mathbb{E}_{11}(\bz,\bz') 
=& \frac{1}{4}\iint_{\partial\Omega\times \partial\Omega} \left[\left(\mathbf{\Gamma}(\bx,\bz)\times\nu(\bx)\right)\times\nu(\bx)\right]^T
\mathbb{E}\left[\overline{\etab(\bx)}\left(\etab(\bx')
\right)^T\right]
\nonumber
\\
&\qquad
\overline{\left(\mathbf{\Gamma}(\bx',\bz')\times\nu(\bx')\right)\times\nu(\bx')}
 d\sigma(\bx)d\sigma(\bx'),
\nonumber
\\
=&
\frac{\sigma_{\rm noise}^2}{4}\int_{\partial\Omega} \left[\mathbf{\Gamma}(\bx,\bz)\right]^T
\overline{\mathbf{\Gamma}(\bx,\bz')} d\sigma(\bx),
\end{align}
where in order to obtain the latter identity, expression \eqref{EE} has been invoked. Assuming, $\bz,\bz'\in\Omega$ far from $\partial \Omega$ and utilizing the Helmholtz-Kirchhoff identities, we obtain 
\begin{align}
\mathbb{E}_{11}(\bz,\bz')
\simeq &
-{\epsilon_0\sigma^2_{\rm noise}}(4\K)^{-1}\Im m\big\{\mathbf{\Gamma}(\bz,\bz')\big\}.\label{EE_1}
\end{align}

Now, remark that  
$$
\left(\nabla\times\left(\phi\mathbf{I}_3\right)\right)^T \mathbf{p}= -\nabla\times\left(\phi\mathbf{p}\right),
$$ 
for any constant vector $\mathbf{p}$ and any smooth function $\phi$. Therefore, 
\begin{align*}
\nabla_\bx\times\left(g(\by,\bx)\etab(\by)\times\nu(\by)\right)
=&
-\left[\nabla_\bx\times\left(g(\by,\bx)\mathbf{I}_3 \right)\right]^T\left(\etab(\by)\times\nu(\by)\right),
\\
=& -{\epsilon_0^{-1}}\left[\nabla_\by\times\mathbf{\Gamma}(\by,\bx)\right]^T\left(\etab(\by)\times\nu(\by)\right),
\\
=&\phantom{-}\,{\epsilon_0^{-1}}\left[\nabla_\by\times\mathbf{\Gamma}(\by,\bx)\times\nu(\by)\right]^T \etab(\by).
\end{align*}
Consequently, for $\by\in\Omega$ far from boundary $\partial\Omega$
\begin{align}
\mathcal{P}^\K\left[\etab\times\nu\right](\bx)=\frac{i\K}{\epsilon_0}\int_{\partial\Omega}
\mathbf{\Gamma}(\by,\bx)\etab(\by) d\sigma(\by)\times\nu(\bx).\label{Pexpanded}
\end{align}

By virtue of \eqref{Pexpanded}, we have
\begin{align}
\mathbb{E}_{12}(\bz,\bz')=& 
-\frac{i\K}{2\epsilon_0}\iiint_{(\partial\Omega)^3} \left[\left(\mathbf{\Gamma}(\bx,\bz)\times\nu(\bx)\right)\times\nu(\bx)\right]^T
\mathbb{E}\left[\overline{\etab(\bx)}\left(\etab(\by)\right)^T\right]
\nonumber
\\
&\qquad
\mathbf{\Gamma}(\by,\bx')\overline{\left[\mathbf{\Gamma}(\bx',\bz')\times\nu(\bx')\right]\times\nu(\bx')}
 d\sigma(\by)d\sigma(\bx)d\sigma(\bx'),
 \nonumber
\\
=&-\frac{i\K\sigma_{\rm noise}^2}{2\epsilon_0}\iint_{(\partial\Omega)^2} \mathbf{\Gamma}(\by,\bz)
\left[\mathbf{\Gamma}(\by,\bx')\right]^T \overline{\mathbf{\Gamma}(\bx',\bz') }d\sigma(\bx')d\sigma(\by).
\end{align}
Invoking Helmholtz-Kirchhoff identity, we get
\begin{eqnarray}
\mathbb{E}_{12}(\bz,\bz')\simeq  \frac{\sigma_{\rm noise}^2}{2\K^2}\int_{\partial\Omega}\Big\{i\K\mathbf{\Gamma}(\by,\bz)\Big\}\Im m\Big\{\K\mathbf{\Gamma}(\by,\bz')\Big\}d\sigma(\by).\label{EE_2}
\end{eqnarray}

Similarly, third term $\mathbb{E}_3(\bz,\bz')$ can be evaluated and appears to be 
\begin{eqnarray}
\mathbb{E}_{21}(\bz,\bz')\simeq \frac{\sigma_{\rm noise}^2}{2\K^2}\int_{\partial\Omega}\Im m\Big\{\K\mathbf{\Gamma}(\by,\bz)\Big\}\overline{\Big\{i\K\mathbf{\Gamma}(\by,\bz')\Big\}}d\sigma(\by).\label{EE_3}
\end{eqnarray}

In order to explicitly calculate $\mathbb{E}_4(\bz,\bz')$, we observe by invoking \eqref{Pexpanded} that 
\begin{align*}
\mathbb{E}\Big[\mathcal{P}^\K&\left[\etab\times\nu\right](\bx)\overline{\mathcal{P}^\K\left[\etab\times\nu\right](\bx')}^T\Big]
\\
=&
\frac{\K^2}{\epsilon_0^2}\mathbb{E}\Bigg[
\iint_{(\partial\Omega)^2}
\mathbf{\Gamma}(\by,\bx) \etab(\by)\times\nu(\bx)
\left[\overline{
\mathbf{\Gamma}(\by',\bx')\etab(\by')\times\nu(\bx')
}\right]^Td\sigma(\by)d\sigma(\by')
\Bigg],
\\
=&
\frac{\K^2}{\epsilon_0^2}\iint_{(\partial\Omega)^2}
\left[\mathbf{\Gamma}(\by,\bx)\times\nu(\bx)\right]^T \mathbb{E}\left[\etab(\by)\overline{\etab(\by')}^T\right]
\overline{\mathbf{\Gamma}(\by',\bx')\times\nu(\bx')}
d\sigma(\by)d\sigma(\by'),
\\
\simeq &
\frac{\sigma_{\rm noise}^2\K^2}{\epsilon_0^2} 
\int_{\partial\Omega}
\left[\mathbf{\Gamma}(\by,\bx)\times\nu(\bx)\right]^T 
\overline{
\left[
\mathbf{\Gamma}(\by,\bx')\times\nu(\bx')\right]
}d\sigma(\by).
\end{align*}

Therefore, 
\begin{align}
\mathbb{E}_{22}(\bz,\bz')=&\iint_{(\partial\Omega)^2} \left(\mathbf{\Gamma}(\bx,\bz)\times\nu(\bx)\right)^T
\mathbb{E}\left[\overline{\mathcal{P}^\K\left[\etab\times\nu\right](\bx)}\left(\mathcal{P}^\K\left[\etab\times\nu\right](\bx')\right)^T\right]
\nonumber
\\
&\qquad\qquad
\overline{\mathbf{\Gamma}(\bx',\bz')\times\nu(\bx')} d\sigma(\bx)d\sigma(\bx'),
\nonumber
\\
\nonumber
\simeq&\frac{\sigma_{\rm noise}^2\K^2}{\epsilon_0^2}\iiint_{(\partial\Omega)^3} \left(\mathbf{\Gamma}(\bx,\bz)\times\nu(\bx)\right)^T
\overline{\left[\mathbf{\Gamma}(\by,\bx)\times\nu(\bx)\right]^T}
\\
\nonumber
&\qquad\qquad
\left[
\mathbf{\Gamma}(\by,\bx')\times\nu(\bx')\right]
\overline{\mathbf{\Gamma}(\bx',\bz')\times\nu(\bx')} d\sigma(\bx)d\sigma(\bx')d\sigma(\by),
\\
\simeq&\frac{\sigma_{\rm noise}^2}{\K^2}\int_{\partial\Omega} \Im m\Big\{\K\mathbf{\Gamma}(\bz,\by)\Big\}\Im m\Big\{\K\mathbf{\Gamma}(\by,\bz')\Big\} d\sigma(\by). \label{EE_4}
\end{align}

Adding all the contributions $\mathbb{E}_{pq}$ (for $p,q\in\{1,2\}$), we obtain the covariance of $\mathbf{U}^{\rm noise}$ 
\begin{align*}
\mathbb{E}\Big[\mathbf{U}^{\rm nosise}(\bz)
\overline{\mathbf{U}^{\rm nosise}(\bz')}^T\Big]
=&
-\frac{\sigma^2_{\rm noise}}{4\K\epsilon_0}\Im m\Big\{\mathbf{\Gamma}(\bz,\bz')\Big\}
\\
&+\frac{\sigma_{\rm noise}^2}{2\K^2\epsilon_0^2}\int_{\partial\Omega}\Big\{i\K\mathbf{\Gamma}(\by,\bz)\Big\}\Im m\Big\{\K\mathbf{\Gamma}(\by,\bz')\Big\}d\sigma(\by)
\\
&+\frac{\sigma_{\rm noise}^2}{2\K^2\epsilon_0^2}\int_{\partial\Omega}\Im m\Big\{\K\mathbf{\Gamma}(\by,\bz)\Big\}\overline{\Big\{i\K\mathbf{\Gamma}(\by,\bz')\Big\}}d\sigma(\by)
\\
&+
\frac{\sigma_{\rm noise}^2}{\K^2\epsilon_0^2}\int_{\partial\Omega} \Im m\Big\{\K\mathbf{\Gamma}(\bz,\by)\Big\}\Im m\Big\{\K\mathbf{\Gamma}(\by,\bz')\Big\} d\sigma(\by).
\end{align*} 
Finally,  the result follows by the fact that for any complex number $Z$
$$
2 \Im m\{iZ\}\Im m\{iZ\}-Z\Im m\{iZ\}-\Im m\{Z\}\overline{Z}=0.
$$

\bibliographystyle{plain}

\begin{thebibliography}{99}

\bibitem[Ammari et al.(2013)]{TDelastic} H. Ammari, E. Bretin, J. Garnier, W. Jing, H. Kang, and A. Wahab, Localization, stability, and resolution of topological derivative based imaging
functionals in elasticity, {\sl SIAM J. Imag. Sci.}, 6(4):(2013), pp. 2174--2212.

\bibitem[Ammari et al.(2015)]{Princeton} H. Ammari, E. Bretin, J. Garnier, H. Kang, H. Lee, and A. Wahab, {\sl Mathematical Methods in Elasticity Imaging}, Princeton Series in Applied Mathematics, Princeton University Press, New Jersey, USA, 2015, ISBN: 978-0-69116531-8.

\bibitem[Ammari et al.(2013)]{ABGW} H. Ammari, E. Bretin, J. Garnier, and A. Wahab, Time reversal algorithms in visco-elastic media, {\sl European J. Applied Mathematics}, 24(4):(2013), pp. 565--600.

\bibitem[Ammari et al.(2014)]{Multi} H. Ammari, J. Garnier, W. Jing, H. Kang, M. Lim, K. S{\o}lna, and H. Wang, {\sl Mathematical and Statistical Methods for Multistatic Imaging}, Lect. Notes Math., Vol. 2098, Springer, 2014.

\bibitem[Ammari et al.(2012)]{AGJK} H. Ammari, J. Garnier, V. Jugnon, and H. Kang, Stability and resolution analysis for a topological derivative based imaging functional, {\sl SIAM J. Cont. Opt.}, 50(1):(2012), pp. 48--76.

\bibitem[Ammeri et al.(2007)]{AILP} H. Ammari, E. Iakovleva, D. Lesselier, and G. Perrusson, MUSIC-type electromagnetic imaging of a collection of small three-dimensional inclusions, {\sl SIAM J. Sci. Comp.}, 29:(2007), pp. 674--709.

\bibitem[Ammari and Kang(2007)]{AKpol} H. Ammari, and H. Kang, {\sl Polarization and Moment Tensors: With Applications to Inverse Problems and Effective Medium Theory}, Appl. Math. Sci. 162, Springer-Verlag, New York, 2007.

\bibitem[Ammari and Kang(2004)]{AK} H. Ammari, and H. Kang, {\sl Reconstruction of Small Inhomogeneities from Boundary Measurements}, Lecture Notes in Mathematics, Vol. 1846, Springer-Verlag, Berlin, 2004.

\bibitem[Ammari and Kang(2003)]{AKInv} H. Ammari, and H. Kang, {A new method for reconstructing electromagnetic
inhomogeneities of small volume}, {\sl Inverse Problems}, 19:(2003), pp.  63--71.

\bibitem[Ammari et al.(2001)]{AVV} H. Ammari, M. Vogelius, and D. Volkov, Asymptotic formulas for perturbations in the electromagnetic fields due to the presence of inhomogeneities of small diameter II. The full Maxwell equations, {\sl J. Math. Pur. Appl.}, 80(8):(2001), pp. 769--814.

\bibitem[Ammari and Volkov(2005)]{AV} H. Ammari, and D. Volkov, The leading-order term in the asymptotic expansion of the scattering amplitude of a collection of finite number of dielectric inhomogeneities of small diameter, {\sl Inter. J. Multi. Comput. Eng.}, 3(3):(2005), pp. 149-160.

\bibitem[Asch and Mefire(2008)]{AM} M. Asch, and S. M. Mefire, Numerical localization of electromagnetic imperfections from a perturbation formula in three dimensions, {\sl J. Comp. Math.}, 26(2): 2008, pp. 149--195.

\bibitem[Bao et al.(2014)]{Bao} G. Bao, J. Lin, and S. M. Mefire, Numerical reconstruction of electromagnetic inclusions in three dimensions, {\sl SIAM J. Imag. Sci.}, 7(1): (2014), pp. 558--577.

\bibitem[Bellis et al.(2013)]{Bellis} C. Bellis, M. Bonnet, and F. Cakoni, Acoustic inverse scattering using topological derivative of far-field measurements-based $L^2$ cost functionals, {\sl Inv. Prob.}, 29(7): (2013), 075012. 

\bibitem[Bonnet and Guzina(2004)]{BG} M. Bonnet and B. B. Guzina, Sounding of finite solid bodies by way of topological derivative, {\sl Internat. J. Numer. Methods Engrg.}, 61:(2004), pp. 2344--2373.

\bibitem[Buffa and Ciarlet(2001)]{BC} A. Buffa, and  P. Ciarlet Jr., On traces for functional spaces related to Maxwell's equations. I. An integration by parts formula in Lipschitz polyhedra, {\sl Math. Meth. Appl. Sci.}, 24:(2001),  pp. 9--30.

\bibitem[Buffa et al.(2003)]{BHPS} A. Buffa, R. Hiptmair, T. von Petersdroff, and C. Schwab, Boundary element methods for Maxwell transmission problems in Lipschitz domain, {\sl Numer. Math.}, 95:(2003), pp. 459--485.

\bibitem[C{\'e}a et al.(2001)]{CGGM} J. C{\'e}a, S. Garreau, P. Guillaume, and M. Masmoudi, The shape and topological optimization connection, {\sl Comput. Methods Appl. Mech. Engrg.}, 188:(2001), pp. 703--726.

\bibitem[Chen et al.(2013)]{CCH} J. Chen, Z. Chen, and G. Huang, Reverse time migration for extended obstacles: electromagnetic
waves, {\sl Inverse Problems}, 29:(2013), 085006.

\bibitem[Colton and Kress(1983)]{colton} D. Colton, and R. Kress, {\sl Integral Equation Methods in Scattering Theory}, Pure and Applied Mathematics, John Wiley \& Sons Inc., New York, 1983. 

\bibitem[Costabel and Le Lou\"{e}r(2012)]{CL} M. Costabel and F. le Lou\"{e}r, Shape derivatives of boundary integral operators in electromagnetic scattering. Part II: Application to scattering by a homogeneous dielectric obstacle, {\sl Integr. Equ. Oper. Theory}, 73 (2012), pp. 17--48.

\bibitem[Dominguez and Gibiat (2010)]{DG} N. Dominguez, and V. Gibiat, Non-destructive imaging using the time domain topological
energy method, {\sl Ultrasonics}, 50:(2010), pp. 172--179.

\bibitem[Dominguez et al.(2005)]{DGE} N. Dominguez, V. Gibiat, and Y. Esquerrea, Time domain topological gradient and time
reversal analogy: An inverse method for ultrasonic target detection, {\sl Wave Motion}, 42:(2005), pp. 31--52.

\bibitem[Eschenauer et al.(1994)]{EKS} A. Eschenauer, V. V. Kobelev, and A. Schumacher, Bubble method for topology and shape optimization of structures, {\sl Struct. Optim.}, 8:(1994), pp. 42--51.

\bibitem[Feij{\'o}o(2004)]{F} G. R. Feij{\'o}o, A new method in inverse scattering based on the topological derivative, {\sl Inverse Problems}, 20:(2004), pp. 1819--1840.

\bibitem[Guzina and Chikichev(2007)]{GuzinaChik} B. B. Guzina, and I. Chikichev, From imaging to material identification: A generalized concept of topological sensitivity, {\sl J. Mech. \& Phys. Solids}, 55:(2007), pp. 245--279.

\bibitem[Hinterm{\"u}ller and Laurain(2008)]{HL} M. Hinterm{\"u}ller, and A. Laurain, Electrical impedance tomography: From topology to shape, {\sl Control Cybernet.}, 37:(2008), pp. 913--933.

\bibitem[Hiptmair(1999)]{H} R. Hiptmair, Symmetric coupling for eddy current problems, {\sl SIAM J. Numer. Anal.}, 40:(2002),
pp. 41--65.


\bibitem[McCamy and Stephan(1984)]{MS} R. McCamy, and  E. Stephan, Solution procedures for three dimensional eddy-current problems, {\sl J. Math. Anal. Appl.}, 101:(1984), pp. 348--379. 


\bibitem[Masmoudi et al.(2005)]{MPS} M. Masmoudi, J. Pommier, and B. Samet, The topological asymptotic expansion for the Maxwell equations and some applications, {\sl Inverse Problems}, 21:(2005), pp. 547--564.

\bibitem[N\'ed\'elec(2001)]{nedelec} J. C. N\'ed\'elec, {\sl Acoustic and Electromagnetic Equations:Integral Representations for Harmonic Problems}, App. Math. Sci., 
Vol. 144, Springer-Verlag, New York, 2001.

\bibitem[Olver et al.(2010)]{NIST} F. W. J. Olver, D. W. Lozier, R. F. Boisvert, and C. W. Clark, editors, {\sl NIST Handbook of
Mathematical Functions}, Cambridge, 2010.

\bibitem[Park(2012)]{park} W.-K. Park, Topological derivative strategy for one-step iteration imaging of arbitrary shaped thin, curve-like electromagnetic inclusions, {\sl J. Comp. Phy.}, 231:(2012), pp. 1426-1439.


\bibitem[Park(2013)]{park2} W.-K. Park, Multi-frequency topological derivative for approximate shape acquisition of curve-like thin electromagnetic inhomogeneities, {\sl J. Math. Anal. App.}, 404(2):(2013), pp. 501--518. 


\bibitem[Soko{\l}owski and {\.Z}ochowski(1999)]{SZ} J. Soko{\l}owski, and A. {\.Z}ochowski, On the topological derivative in shape optimization, {\sl SIAM J. Control. Optim.}, 37:(1999), pp. 1251--1272.

\bibitem[Souhir et al.(2012)]{GWL} S. Gdoura, A. Wahab, and D. Lesselier, Electromagnetic time reversal and scattering by a small dielectric inclusion, {\sl Journal of Physics: Conference Series}, 386:(2012), 012010.

\bibitem[Wahab et al.(2014)]{WAHR} A. Wahab, A. Rasheed, T. Hayat, R. Nawaz, Electromagnetic time reversal algorithms and source localization in lossy dielectric media, {\sl Communications in Theoretical Physics}, 62(6):(2014), pp. 779--789.

\bibitem[Wahab et al.(2014)]{WARS} A. Wahab, A. Rasheed, R. Nawaz, and S. Anjum, Localization of extended current source with finite frequencies, {\sl Comptes Rendus Math\'ematique}, 352:(2014), pp. 917--921.


\end{thebibliography}

\end{document}